\documentclass[12pt,reqno]{amsart}
\usepackage{amsmath, amssymb, amsfonts, amsthm}
\usepackage{mathtools,xcolor}
\usepackage{hyperref}
 \usepackage{mathrsfs}
\newtheorem{thm}{Theorem}[section]
\newtheorem{lem}[thm]{Lemma}
\newtheorem{prop}[thm]{Proposition}
\theoremstyle{definition}
\newtheorem{dfn}[thm]{Definition}

\newtheorem{remark}[thm]{Remark}
\theoremstyle{plain}

\numberwithin{equation}{section}
\usepackage[a4paper, top = 1.2in, bottom = 1.2in, left = 1.2in, right = 1.2in]{geometry}
\numberwithin{equation}{section}

\newcommand{\Q}{\mathbb{Q}}

\newcommand{\Z}{\mathbb{Z}}

\newcommand{\F}{\mathbb{F}}

\newcommand{\mfa}{\mathfrak{a}}

\newcommand{\mfp}{\mathfrak{p}}
\newcommand{\mfl}{\mathfrak{l}}

\newcommand{\ch}{\mathrm{ch}}
\newcommand{\SL}{\mathrm{SL}}
\newcommand{\PSL}{\mathrm{PSL}}
\newcommand{\GL}{\mathrm{GL}}
\newcommand{\GU}{\mathrm{GU}}
\newcommand{\GO}{\mathrm{GO}}
\newcommand{\PGL}{\mathrm{PGL}}

\newcommand{\End}{\mathrm{End}}
\newcommand{\Aut}{\mathrm{Aut}}
\newcommand{\Frob}{\mathrm{Frob}}
\newcommand{\tr}{\mathrm{tr}}

\newcommand{\Spec}{\mrm{Spec}}
\newcommand{\sep}{\mrm{sep}}
\newcommand{\un}{\mrm{un}}

\newcommand{\Gal}{\mrm{Gal}}
\newcommand{\Ima}{\mrm{Im }}

\newcommand{\Sp}{\textrm{Sp}}

\newcommand{\lra}{\longrightarrow}
\newcommand{\ra}{\rightarrow}
\newcommand{\ras}{\twoheadrightarrow}
\newcommand{\mrm}[1]{\mathrm{#1}}

\def\1{1\!\!1}

\newcommand{\psmat}[4]{\bigl( \begin{smallmatrix} #1 & #2 \\ #3 & #4 \end{smallmatrix} \bigr)}

\newcommand{\csmat}[9]{\biggl( \begin{smallmatrix} #1 & #2 & #3 \\
		#4 & #5 & #6 \\ #7 & #8 & #9 \end{smallmatrix} \biggr)}

\title[On the surjectivity of Galois representations]{A class of Drinfeld $A$-modules of rank $3$ with surjective Galois representations}
\author[N. Kumar]{Narasimha Kumar}
\email{narasimha@math.iith.ac.in}
\address{
	Department of Mathematics \\
	Indian Institute of Technology Hyderabad\\
	Kandi, Sangareddy - 502284\\
	INDIA.
}
\author[D. Shit]{Dwipanjana Shit}
\email{ma22resch01001@iith.ac.in}
\address{
	Department of Mathematics \\
	Indian Institute of Technology Hyderabad\\
	Kandi, Sangareddy - 502284\\
	INDIA.
}

\begin{document}
	\begin{abstract}
		Let $q = p^e \geq 7$ be an odd prime power, and set $A := \F_q[T]$. 
 		In this article, we construct an infinite two-parameter family of Drinfeld $A$-modules of rank $3$ such that, for every non-zero prime ideal $\mfl$ of $A$, the associated mod-$\mfl$, $\mfl$-adic, and adelic Galois representations are surjective. These results generalise the specific example, constructed only for primes $p\equiv 1\pmod{3}$, in~\cite{Che22}.
		
	\end{abstract}
	\maketitle
	\section{Introduction}
	In his seminal work~\cite{Ser72}, Serre investigated adelic images of elliptic curves over $\Q$
	without complex multiplication. He proved the following:
	\begin{thm}[\cite{Ser72}]
		If $E$ is an elliptic curve over $\mathbb{Q}$ without CM, then the associated adelic Galois representation
		$\rho_{E}: \Gal(\bar{\Q}/\Q) \ra \varprojlim_{m}\ \Aut(E[m]) \cong \GL_{2}(\widehat{\Z})$
		has open image. In particular, $[\GL_{2}(\widehat{\Z}): \Ima(\rho_{E})] < \infty.$ 
	\end{thm}
	Analogously, Pink and R\"utsche~\cite{PR09} studied adelic images in the context of Drinfeld $A$-modules, proving the following:
	\begin{thm}[\cite{PR09}]
		Let $\varphi$ be a Drinfeld $A$-module of rank $r$ such that $\varphi$ is generic.
		If $\End_{\bar{F}}(\varphi)=\varphi(A)$, then the image of the associated adelic Galois representation $\rho_{\varphi}(G_{F})$ is open in $\GL_{r}(\widehat{A})$. Equivalently,  $[\GL_{r}(\widehat{A}):\rho_{\varphi}(G_{F})]<\infty.$
	\end{thm}
	
	A natural problem in both settings is determining when the index of the image is equal to one; that is, when the adelic Galois representation is surjective. For elliptic curves over $\Q$, in~\cite{Ser72}, Serre showed that the answer is negative, i.e., the image has index at least $2$ in $\GL_{2}(\widehat{\Z})$. However, in~\cite{Gre10}, Greicius gave an example of a number field $K$ and an elliptic curve $E/K$ that admits a surjective adelic Galois representation.

	In contrast, for Drinfeld $A$-modules, there are known examples with surjective adelic representations. In~\cite{Hay74}, Hayes showed that the adelic Galois representation of the Carlitz $A$-module is surjective. 
	In~\cite{Zyw11}, Zywina constructed Drinfeld $A$-modules of rank $2$ defined by $\varphi_{T}=T+\tau-T^{q-1}\tau^2$ having surjective adelic Galois representations 
	if $q \geq 5$ is an odd prime power. Inspired by~\cite{Zyw11}, Chen \cite{Che22} constructed a Drinfeld $A$-module of rank $3$ with surjective adelic Galois representation. More precisely, he showed that
	
	\begin{thm}[\cite{Che22}, Theorem 1]\label{chen_rank_3_case}
		Let $q=p^e$ be a prime power with $p\geq 5$ and $p\equiv 1\pmod{3}$. Let $\varphi$ be the Drinfeld $A$-module of rank $3$ defined by $\varphi_{T}=T+\tau^2+T^{q-1}\tau^3$. The adelic Galois representation
		$$\rho_{\varphi}:G_{F} \lra \varprojlim_{\mfa}\ \Aut(\varphi[\mfa])\cong\GL_{3}(\widehat{A})$$
		is surjective.
	\end{thm}
	In this work, we extend Chen’s result, i.e., Theorem~\ref{chen_rank_3_case}, by constructing an infinite class of rank $3$ Drinfeld $A$-modules with surjective adelic representations. The main ideas of this article are inspired by~\cite{Zyw11} and~\cite {Che22}. Let $\mathcal{G}\subseteq A^2$ be an infinite family of pairs $(g_{1},g_{2})$ such that the following holds:
	\begin{itemize}
		\item $(g_1, T^{q}-T)=T$ (Type $1$) or $T^{q}-T$ (Type $2$).
		\item $(g_2, T^{q}-T)=1$.
	\end{itemize}
	Note $(0,1) \in \mathcal{G}$. Our main theorem can now be stated as follows.
	\begin{thm}[= Theorem~\ref{Duplicate_l_adic_surjectivity}]
		\label{l_adic_surjectivity}
		Let $q\geq 7$ be an odd prime power. Let $\varphi$ be the Drinfeld $A$-module  of rank $3$ defined by $\varphi_{T}=T+g_{1}^{q-1}\tau+g_{2}^{q-1}\tau^2+T^{q-1}\tau^3$ with $(g_{1},g_{2})\in \mathcal{G}$. 
		The $\mfl$-adic Galois representation
		$$\rho_{\varphi, \mfl}:G_{F} \longrightarrow \varprojlim_{i}\ \Aut(\varphi[\mfl^i]) \cong \GL_{3}(A_{\mfl})$$
		is surjective, for all $\mfl \in \Omega_{A}$.
	\end{thm}
	As a consequence of Theorem~\ref{l_adic_surjectivity} with some suitable modifications of arguments in~\cite[\S8]{Che22}, we get the surjectivity of the adelic representation (cf. Theorem~\ref{adelic_surjectivity}). 
	
	\subsection{Organization}
	Following the introductory section, §2 discusses Drinfeld modules, their associated Galois representations, and Tate uniformization.
	In $\S3$, for any non-zero prime ideal $\mfl$ of $A$, we establish the irreducibility of the mod-$\mfl$ representations. In $\S 4$, we provide estimates for the size of the image of the mod-$\mfl$ representation. In §5, we apply Aschbacher’s theorems [BHR13, Theorems 2.2.19 and 4.10.2] to prove the surjectivity of the mod-$\mfl$ representation. The main result of the article is proven in §6, where we demonstrate the surjectivity of the $\mfl$-adic representation, using [PR09, Proposition 4.1]. Consequently, we establish the surjectivity of the adelic representation in $\S7$. Finally, in §8, we compare our results with those of~\cite{Che22}.

	\section{Preliminaries}
	Throughout this article, we stick to the following notations. Let $q\geq 7$ be an odd prime power, unless explicitly stated otherwise. Let $A:=\F_{q}[T]$ with field of fraction $F:=\F_{q}(T)$, $G_{F}:=\Gal(F^{\sep}/F)$, where $F^{\sep}$ is the separable closure of $F$ in $\bar{F}$.  Let $\widehat{A} :=\varprojlim_{\mfa \lhd A}A/\mfa$, where $\mfa$ runs over all non-zero ideals of $A$, denote the profinite completion of $A$. For a commutative ring $R$ with unity, $R^{\times}$ denotes the set of all units in $R$.
	Let $\Spec(A)$ be the set of all prime ideals in $A$ and $\Omega_{A}:=\Spec(A)\setminus \{(0)\}$. To avoid notational complexity, we shall use $\mfa\subseteq A$ to denote both the generator and the ideal it generates. Let $\mfp\in \Omega_{A}$, $A_{\mfp}$ denote the completion of $A$ with respect to $\mfp$ with  field of fraction $F_{\mfp}$, which is complete with respect to the normalized discrete valuation $\nu_p$, with the residue field $\F_{\mfp} := A_\mfp/\mfp A_\mfp$.

	
	

	\subsection{Tate uniformization}
	\begin{dfn}
		Let $\mfp\in \Omega_{A}$ and $\psi:A\ra A_{\mfp}\{\tau\}$ be a Drinfeld $A_{\mfp}$-module. A $\psi$-lattice of rank $d$ is a free $A$-submodule $\Lambda\subseteq {}^{\psi}F_{\mfp}^{\sep}$ of rank $d$, which is invariant under the action of $G_{F_{\mfp}}$ and discrete with respect to the topology of the local field  $F_{\mfp}^{\sep}$. Here, ${}^{\psi}F_{\mfp}^{\sep}$ denotes the $A$-action on $F_{\mfp}^{\sep}$ induced via $\psi$.
	\end{dfn}
	
	Let $r$ and $d$ be two positive integers. A Tate datum over $A_{\mfp}$ is a pair $(\psi,\Lambda)$, where $\psi$ is a Drinfeld $A_{\mfp}$-module of rank $r$  and $\Lambda$ is a $\psi$-lattice of rank $d$. Two such Tate datum $(\psi,\Lambda)$ and $(\psi^\prime,\Lambda^\prime)$ are isomorphic if there is an isomorphism from $\psi$ to $\psi^\prime$ such that the induced $A$-module homomorphism ${}^{\psi}F_{\mfp}^{\sep}\ra {}^{\psi^\prime}F_{\mfp}^{\sep}$ gives an $A$-module isomorphism $\Lambda \ra \Lambda^\prime$.  The correspondence below is well-known, and we refer to it as the Drinfeld-Tate uniformization (cf.~\cite[Theorem 6.2.11]{Pap23}).
	
	\begin{thm}[\cite{Dri74}, \S 7]\label{Drinfeld}
		Let $r,d$ be two positive integers. There is a one-to-one correspondence between the following two sets:
		\begin{enumerate}
			\item The set of $F_{\mfp}$-isomorphism classes of Tate datum $(\psi,\Lambda)$ where $\psi$ is a Drinfeld $A_{\mfp}$-module of rank $r$ with good reduction, $\Lambda$ is a $\psi$-lattice of rank $d$.
			\item The set of $F_{\mfp}$-isomorphism classes of Drinfeld modules $\varphi:A\ra A_{\mfp}\{\tau\}$ of rank $r+d$ with stable reduction of rank $r$.
		\end{enumerate}
	\end{thm}
	In the proof of the above correspondence, one uses an entire $\F_{q}$-linear function denoted by $e_{\Lambda}(x)$, and defined by
	\begin{align}\label{e_Lambda_product_expression}
		e_{\Lambda}(x) := x \prod_{\substack{\lambda \in \Lambda \\ \lambda \ne 0}} \left(1 - \frac{x}{\lambda}\right) .
	\end{align}
	Here, we recall some important properties of $e_{\Lambda}(x)$:
	\begin{itemize}
		\item  The function $e_{\Lambda}(x)$ satisfies the relation $e_{\Lambda}    (\psi_{T}(x))=\varphi_{T}(e_{\Lambda}(x))$. 
		\item  The function $e_{\Lambda}(x)$ has a power series expansion as
		\begin{align}\label{e_Lambda_power_series_expression}
			e_{\Lambda}(x)=u_{0}x+u_{1}x^{q}+u_{2}x^{q^2}+\cdots+u_{n}x^{q^n}+\cdots
		\end{align}
		where $u_{0}=1$, and $u_{n}\in \mfp A_{\mfp}$ for $n\geq 1$ with $\nu_{\mfp}(u_n)\to \infty$ as $n\to \infty$.
		\item  If $\Lambda$ is a $\psi$-lattice of rank $1$, then by~\eqref{e_Lambda_product_expression}, we have
		\begin{align}\label{e_Lambda_for_rank_1_lattice}
			e_{\Lambda}(x)=x\cdot\prod_{\substack{0\neq a \in A}} \left(1 - \frac{x}{\psi_{a}(\lambda)} \right),
		\end{align}
		where $\lambda$ is a generator of $\Lambda.$
	\end{itemize}
	Let $\varphi$ be a generic Drinfeld $A$-module. For any $a\in A\setminus \{0\}$, the $a$-torsion of $\varphi$, denoted by $\varphi[a]$, is the set of all roots of the $\F_q$-linear separable polynomial $\varphi_a(x)$ in $F^{\sep}$. For any non-zero ideal $\mfa\subseteq A$, $\varphi[\mfa]:=\varphi[a]$ where $a$ is a generator of $\mfa$. 
	Continuing the discussion above, for any $a \in A \setminus \F_q$, we have the following: 
	\begin{enumerate}
		\item There exists a short exact sequence of $A[G_{F_{\mfp}}]$-modules
		\begin{equation}
			\label{s_e_s_drinfeld_datum}
			0 \longrightarrow \psi[a] = \psi_a^{-1}(0) \longrightarrow \varphi[a] \xrightarrow{\psi_a} \Lambda/a\Lambda \longrightarrow 0.
		\end{equation}
		\item There exists an $A[G_{F_{\mfp}}]$-module isomorphism
		\begin{align}
			\label{iso_drinfeld_datum}
			e_{\Lambda}: \psi_{a}^{-1}(\Lambda)/\Lambda  &\xrightarrow{\sim} \varphi[a]\\
			z+\Lambda &\mapsto e_{\Lambda}(z).\notag
		\end{align}
		\item We have
		\begin{align}\label{expression_of_varphi_mfl_in_terms_of_e_Lambda}
			\varphi_{a}(x)=ax\cdot\prod_{\substack{0\neq \gamma^\prime \in \psi_{a}^{-1}(\Lambda)/\Lambda}}\left(1 - \frac{x}{e_{\Lambda}(\gamma^\prime)} \right).
		\end{align}
	\end{enumerate}
	We now introduce various Galois representations associated to Drinfeld modules. 
	\subsection{Galois representations}
	Let $\varphi$ be a generic Drinfeld $A$-module of rank $r$. By~\cite[Corollary 3.5.3]{Pap23}, 
	for any non-zero ideal $\mfa\subseteq A$, $\varphi[\mfa]$ is isomorphic to $(A/\mfa)^r$ as an $A/\mfa$-module, hence $\Aut_{A}(\varphi[\mfa])\cong \GL_{r}(A/\mfa)$. 
	
	\begin{itemize}
		\item The mod-$\mfa$  Galois representation of $\varphi$ is defined by
		$$\bar{\rho}_{\varphi,\mfa}:G_{F}\longrightarrow \Aut_{A}(\varphi[\mfa]))\cong \GL_{r}(A/\mfa).$$
		
		\item For any $\mfp \in \Omega_A$, the $\mfp$-adic Galois representation of $\varphi$ is defined by
		$${\rho}_{\varphi,\mfp}:G_{F}\longrightarrow \varprojlim_{i}\ \Aut_A(\varphi[\mfp^i])\cong\Aut_{A_{\mfp}}(T_{\mfp}(\varphi))\cong \GL_{r}(A_{\mfp}),$$  where $T_{\mfp}(\varphi)$ denotes the $\mfp$-adic Tate module of $\varphi$. Consequently, we have the adelic Galois representation $\rho_{\varphi}:G_{F}\rightarrow \GL_{r}(\widehat{A})$ associated to $\varphi$.
		
	\end{itemize}
	\subsection{On surjectivity of the Carlitz module}
	The Carlitz module is an example of a Drinfeld $A$-module of rank $1$, which is denoted by $C$, is defined by $C_{T}=T+\tau$. Hayes proved that the adelic Galois representation of the Carlitz module is surjective. More precisely:
	\begin{prop}[\cite{Hay74}]\label{Hayes}
		For every non-zero ideal $\mathfrak{a}$ of $A$, the representation
		$$\bar{\rho}_{C,\mathfrak{a}}:G_{F}\longrightarrow \Aut(C[\mathfrak{a}])\cong (A/\mathfrak{a})^{\times}$$
		is surjective. The representation $\bar{\rho}_{C,\mathfrak{a}}$ is unramified at all finite places of $F$ not dividing $\mathfrak{a}$, and for each monic irreducible polynomial $\mfp$ of $A$ not dividing $\mathfrak{a}$, we have $\bar{\rho}_{C,\mathfrak{a}}(\Frob_{\mfp})\equiv \mfp \pmod{\mathfrak{a}}$. In particular, the adelic representation $\rho_C:G_{F}\ra \GL_{1}(\widehat{A})={\widehat{A}}^{\times}$ is surjective.
	\end{prop}
	\begin{lem}\label{rank_r_rank_1_correspondence_with_Hayes_gives_det_of_mod_a_surejective}
		For any $(g_{1},g_{2})\in A^2$, let $\varphi$ be a Drinfeld $A$-module of rank $3$, defined by $$ \varphi_{T}=T+g_{1}\tau+g_{2}\tau^2+T^{q-1}\tau^3.$$ For every non-zero ideal $\mathfrak{a}$ of $A$, $\det\bar{\rho}_{\varphi,\mfa}(G_{F})=\bar{\rho}_{C,\mfa}(G_{F})=(A/\mfa)^{\times}$.
	\end{lem}
	\begin{proof}
		From~\cite[Theorem 3.7.1(1)]{Pap23}, for every non-zero ideal $\mathfrak{a}$ of $A$, we have $\det\bar{\rho}_{\varphi,\mfa}(G_{F})=\bar{\rho}_{\psi,\mfa}(G_{F})$ where $\psi_{T}=T+T^{q-1}\tau$. Since $T\psi_{T}=C_{T}T$, then $\psi$ and $C$ are isomorphic over $F$. Hence, by Proposition~\ref{Hayes}, we have $\det\bar{\rho}_{\varphi,\mfa}(G_{F})=\bar{\rho}_{\psi,\mfa}(G_{F})=\bar{\rho}_{C,\mfa}(G_{F})=(A/\mfa)^{\times}$ for every non-zero ideal $\mathfrak{a}$ of $A$.
		
	\end{proof}
	
	\section{On the irreducibility of mod-$\mfl$ representations $\bar{\rho}_{\varphi,\mfl}$}
	From now on, we consider a two-parameter family  of Drinfeld $A$-modules of rank $3$ which are defined by 
	$$\varphi_{T}=T+g_{1}^{q-1}\tau+g_{2}^{q-1}\tau^2+T^{q-1}\tau^3$$
	with $(g_{1},g_{2})\in \mathcal{G}$. We start with the following Lemma, which describes the image of the inertia subgroup $I_T$ of $G_F$ at $(T)$ under the mod-$\mfl$ Galois representations $\bar{\rho}_{\varphi,\mfl}$, where $\mfl\in \Omega_{A}$.
	
	\begin{lem}\label{Im_of_inertia_l_noteq_T}
		Suppose $\mfl\neq (T)$. Then, there is a basis of $\varphi[\mfl]$ such that $\bar{\rho}_{\varphi,\mfl}(I_{T})$ is contained in the set $\left\{
		\psmat{I_{2}}{*}{0}{1}\right\}$, $I_2$ denotes a $2 \times 2$ identity matrix.
	\end{lem}
	\begin{proof}
		The Drinfeld module $\varphi_{T}=T+g_{1}^{q-1}\tau+g_{2}^{q-1}\tau^2+T^{q-1}\tau^3$ has stable reduction at $(T)$ of rank $2$, because $T\nmid g_{2}$. By Theorem~\ref{Drinfeld}, the corresponding Drinfeld datum $(\psi,\Lambda)$, where $\psi:A\ra A_{(T)}\{\tau\}$ is a Drinfeld $A_{(T)}$-module of rank $2$ with good reduction at $(T)$ and $\Lambda$ is a $\psi$-lattice of rank $1$. By \eqref{iso_drinfeld_datum}, it is enough to find a basis of $ \psi_{\mfl}^{-1}(\Lambda)/\Lambda$ such that the action of $I_{T}$ on $ \psi_{\mfl}^{-1}(\Lambda)/\Lambda$ is of the form $ \left\{ \psmat{I_{2}}{*}{0}{c} : c\in  \F_{\mfl}^{\times}\right\}$.

		The Drinfeld $A_{(T)}$-module $\psi$ is of rank $2$ and generic, by~\cite[Corollary 3.5.3]{Pap23}, there is a $\F_{\mfl}$-basis $\{w_{1},w_{2}\}$ of $\psi[\mfl]$. Since $\psi$ 
		has good reduction at $(T)$ and $\mfl\neq (T)$, the Galois representation $\bar{\rho}_{\psi,\mfl}: G_{F_{(T)}}\ra\Aut(\psi[\mfl])$ is unramified at $(T)$. In particular, $\sigma(w_{i})=w_{i}$ for all $\sigma\in I_{T}$ and for $i\in \{1,2\}$.   
		
		Since $\Lambda$ is a free $A$-module of rank $1$, we may fix a generator $\lambda$ of $\Lambda$. We can choose $z\in F_{(T)}^{\sep}$ such that $\psi_{\mfl}(z)= \lambda$. Since $\Lambda$ is stable under the Galois action of $G_{F_{(T)}}$, there is a character
		$\chi_{\Lambda}:G_{F_{(T)}}\ra A^{\times}=\F_{q}^{\times}$ such that $\sigma(\lambda)=\chi_{\Lambda}(\sigma)\lambda$ for all $\sigma\in G_{F_{(T)}}$. Since $\psi_{\mfl}$ is compatible with the action of $G_{F_{(T)}}$, by~\eqref{s_e_s_drinfeld_datum}, we have
		$$\psi_{\mfl}(\sigma(z))=\sigma(\psi_{\mfl}(z))=\sigma(\lambda)=\chi_{\Lambda}(\sigma)\lambda=\chi_{\Lambda}(\sigma)\psi_{\mfl}(z)=\psi_{\mfl}(\chi_{\Lambda}(\sigma)z).$$
		Thus, $\sigma(z)-\chi_{\Lambda}(\sigma)z\in \psi[{\mfl}]$, therefore there are some elements $b_{\sigma,1},b_{\sigma,2}$ in $\F_{\mfl}$ such that
		$\sigma(z)-\chi_{\Lambda}(\sigma)z=b_{\sigma,1}w_{1}+b_{\sigma,2}w_{2}$, i.e., $\sigma(z)=b_{\sigma,1}w_{1}+b_{\sigma,2}w_{2}+\chi_{\Lambda}(\sigma)z$. Therefore, the action of $\sigma\in I_{T}$ on $ \psi_{\mfl}^{-1}(\Lambda)/\Lambda$ with respect to the basis $\{w_{1}+\Lambda, w_{2}+\Lambda, z+\Lambda\}$ is of the form
		$\csmat{1}{0}{b_{\sigma,1}}{0}{1}{b_{\sigma,2}}{0}{0}{\chi_{\Lambda}(\sigma)}.$

		Since $\mfl\neq (T)$ and $C$ has a good reduction at $(T)$, the representation $\bar{\rho}_{C,\mfl}$ is unramified at $(T)$, i.e., $\bar{\rho}_{C,\mfl}(I_{T})=1$. Hence, we get $\det\bar{\rho}_{\varphi,\mfl}(I_{T})=1$, by  
		Lemma~\ref{rank_r_rank_1_correspondence_with_Hayes_gives_det_of_mod_a_surejective}.
		Therefore, for any $\sigma\in I_{T}$, we have $\chi_\Lambda(\sigma)=1$. This proves the Lemma. 
	\end{proof}

	\begin{thm}\label{varphi_mfl_irreducible_F_mfl_G_F_module}
		Let $\mfl\in \Omega_{A}$. The $\F_{\mfl}[G_{F}]$-module $\varphi[\mfl]$ is irreducible.
	\end{thm}
	\begin{proof}
		On the contrary, $\varphi[\mfl]$ is reducible for some $\mfl \in \Omega_A$. Then, there is a proper $\F_{\mfl}[G_{F}]$-submodule $X$ of $\varphi[\mfl]$ such that
		$X$ has $\F_{\mfl}$-dimension $1$ or $\F_{\mfl}$-codimension $1$. Hence, there is a basis of $\varphi[\mfl]$
		such that the action of $G_{F}$ on $\varphi[\mfl]$ looks either 
		$$\bullet\psmat{\chi}{*}{0}{B}\ \text{if}\  X \ \text{has dimension}\ 1, \ \textrm{or}\  \quad \bullet\psmat{B}{*}{0}{\chi}\ \text{if}\ X\ \text{has codimension}\ 1, $$
		where $B:G_{F}\ra \GL_{2}(\F_{\mfl})$ is a homomorphism and $\chi:G_{F}\ra \F_{\mfl}^{\times}$ is a character. We consider separately two possible cases.
		
		$\bullet$ Assume $\mfl=(T)$. Let $\mfp=(T-c)\in \Omega_{A}\setminus\{(T)\}$ for some $c \in \F_q^\times$. Since $\mfp \neq (T)$ and $\varphi$ has a good reduction at $\mfp$, $\rho_{\varphi,T}$ is unramified at $\mfp$. So, the matrix $\rho_{\varphi,T}(\Frob_{\mfp})\in \GL_{3}(A_{(T)})$ is well-defined up to conjugation. Let $P_{\varphi,\mfp}(x)=\det(xI_{3}-\rho_{\varphi,T}(\Frob_{\mfp}))$ be the characteristic polynomial of the Frobenius element $\Frob_{\mfp}$.

		By~\cite[Theorem 4.2.7 (2,3)]{Pap23}, we have $P_{\varphi,\mfp}(x)=-\mfp+a_{2}x+a_{1}x^2+x^3\in A[x]$, where $a_{1},a_{2}\in \F_{q}$. Since $P_{\varphi,\mfp}(x)$ is also the characteristic polynomial of Frobenius endomorphism of
		$\varphi\otimes\F_{\mfp}$, the reduction of $\varphi$ modulo $\mfp$, acting on $T_{(T)}(\varphi\otimes\F_{\mfp})$, we have
		\begin{equation}\label{ch_mod_l_eq_T}
			-(\varphi\otimes\F_{\mfp})_{T-c}+(\varphi\otimes\F_{\mfp})_{a_{2}}\tau+(\varphi\otimes\F_{\mfp})_{a_{1}}\tau^2+\tau^3=0.
		\end{equation}
		Since $\varphi_{T}=T+g_{1}^{q-1}\tau+g_{2}^{q-1}\tau^2+T^{q-1}\tau^3$, we have $(\varphi\otimes\F_{\mfp})_{T}=c+g_{1}(c)^{q-1}\tau+g_{2}(c)^{q-1}\tau^2+c^{q-1}\tau^3=c+g_{1}(c)^{q-1}\tau+\tau^2+\tau^3$ as $c, g_{2}(c)$ are in $\F_{q}^{\times}$. Therefore $(\varphi\otimes\F_{\mfp})_{T-c}=g_{1}(c)^{q-1}\tau+\tau^2+\tau^3$. Hence, by~\eqref{ch_mod_l_eq_T}, $a_{2}=g_{1}(c)^{q-1}$ and $a_{1}=1$. So we have $P_{\varphi,\mfp}(x)=-\mfp+g_{1}(c)^{q-1} x+x^2+x^3\in A[x]$. Since the characteristic polynomial of $\bar{\rho}_{\varphi,T}(\Frob_{\mfp})$ is congruent to $P_{\varphi,\mfp}(x)$ modulo $T$, the  characteristic polynomial of $\bar{\rho}_{\varphi,T}(\Frob_{\mfp})$ is $c+g_{1}(c)^{q-1} x+x^2+x^3\in A/(T)[x]=\F_{q}[x]$ for all $c\in \F_{q}^{\times}$. Before proceeding further, we recall a result from the theory of permutation polynomials.
		\begin{prop}[\cite{MS87}, Corollary 2.9]\label{PP_by_MS}
			Let $\F_{q}$ be a finite field with the characteristic $p$ different from $3$.  Then $f(x)=a_3x^3+a_2x^2+a_1x+a_0\in \F_{q}[x]$ with $a_3\neq 0$, permutes $\F_q$ if and only if $a_2^2=3a_3a_1$ and $q\equiv 2\pmod{3}$.
		\end{prop}

		Assume that $g_1$ is of Type $1$. For any $c \in \F_q$, define $M_c(x) := x^3+x^2+x+c$.
		For any $c\in \F_{q}^{\times}$, the characteristic polynomial of $\bar{\rho}_{\varphi, T}(\Frob_{(T-c)})$ is $M_c(x)$. Since
		$\phi[\mfl]$ is reducible, $ M_c(x)$ is reducible for all $c\in \F_{q}^{\times}$. 
		Let $\eta : \F_p \ra \F_p $ be the map defined by $\eta(x)=x^3+x^2+x$. Note that, for any $c \in \F_q$, the polynomial $M_{c}(x)$ of degree $3$ is reducible if and only if $-c\in \Ima(\eta)$. In particular, $M_{c}(x)$ is reducible for all $c \in \F_q$ if and only if the map $\eta$ is bijective, i.e., the polynomial $x^3+x^2+x$ permutes $\F_q$.  
		
		If the characteristic of $\F_q$ is $3$, the map $\eta $ is not one-to-one,  as $0$ and $1$ map to $0$. If the characteristic of $\F_q$ is $\geq 5$, the polynomial $x^3+x^2+x$ cannot permute $\F_q$,
		by Proposition~\ref{PP_by_MS}. Hence,  we get a contradiction. Therefore, $\phi[\mfl]$ is irreducible.
		

		Assume that $g_1$ is of Type $2$. For any $c \in \F_q$, define $M_c(x) := x^3+x^2+c$.
		For any  $c\in \F_{q}^{\times}$, the characteristic polynomial of $\bar{\rho}_{\varphi,T}(\Frob_{(T-c)})$ is $M_c(x)\in \F_{q}[x]$. 
		Since
		$\phi[\mfl]$ is reducible, the polynomial $ M_c(x)$ is reducible for all $c\in \F_{q}^{\times}$. Now arguing as in the previous case, the map $\eta:\F_{q}\ra\F_{q}$ defined by $\eta(x)=x^3+x^2$ is bijective, the polynomial $x^3+x^2$ permutes $\F_q$. But this is not true, since $0$ and $-1$ map to $0$, i.e., $\eta$ is not one-to-one.
		Hence, in this case also $\phi[\mfl]$ is irreducible.
		
		$\bullet$ Assume $\mfl\neq (T)$. Recall that $B:G_{F}\ra \GL_{2}(\F_{\mfl})$ is a homomorphism. Now, arguing as in~\cite[Page 108]{Che22}, we get that either
		$\det B$ or $\chi$ is unramified at every $\mfp\in \Omega_{A}$. Again, the proof splits into two cases.
		
		Suppose $\det B$ is unramified at every prime $\mfp\in \Omega_{A}$. By the discussion in~\cite[Page 109]{Che22}, we can write $\det B$ as
		$$\det B:G_{F}\twoheadrightarrow \Gal(\bar{\F}_q(T)/\F_q(T))\cong\Gal(\bar{\F}_{q}/\F_{q})\ra\F_{\mfl}^{\times}$$
		where the first map is the restriction map. This implies that there is some element $\zeta\in\F_{\mfl}^{\times}$ such that $(\det B)(\Frob_{\mfp})=\zeta^{\deg\mfp}$ for every $\mfp\in \Omega_{A}\setminus\{(T),\mfl\}$. In fact, the image of $\Frob_{\mfp}$ mapped into $\Gal(\bar{\F}_{q}/\F_{q})$ is equal to $\pi^{\deg\mfp}$, where $\pi$ is the Frobenius endomorphism. Thus $\zeta\in \F_{\mfl}$ does not depend on primes $\mfp\in \Omega_{A}\setminus\{(T),\mfl\}$, because $\zeta$ is the image of $\pi$ into $\F_{\mfl}^{\times}$.
		
		By~\cite[Corollary 9]{Che22}, for $\mfp\in \Omega(A)$ such that $\mfp\nmid \mfa$, we have $\det\circ \bar{\rho}_{\varphi,\mfa}(\Frob_{\mfp})\equiv \mfp \pmod \mfa$. Since $(\det B)(\Frob_{\mfp})=\zeta^{\deg\mfp}$, we have $\chi(\Frob_{\mfp})=\zeta^{-\deg\mfp}\bar{\mfp}$ for all $\mfp\in \Omega_{A}\setminus\{(T),\mfl\}$, where $\bar{\mfp}$ denotes the image of $\mfp$ in $\F_{\mfl}$.
		
		We now compute the characteristic polynomial of the Frobenius element $\Frob_\mfp$ for
		$\mfp = (T-c)\in \Omega_{A}\setminus\{(T),\mfl\}$. Arguing as in the case for $\mfl=(T)$,
		we have $P_{\varphi,\mfp}(x)=-\mfp+g_{1}(c)^{q-1} x+x^2+x^3\in A[x]$. Since $P_{\varphi,\mfp}(x)$ is congruent to the characteristic polynomial of $\bar{\rho}_{\varphi,\mfl}(\Frob_\mfp)$ to modulo $\mfl$, we have a factorization
		\begin{align}\label{factorization_1}
			-\bar{\mfp}+g_{1}(c)^{q-1}  x+x^2+x^3=(x^2-\alpha_{\mfp}x+\zeta)(x-\zeta^{-1}\bar{\mfp})\in \F_{\mfl}[x].
		\end{align}
		
		Now two cases arise depending on the choice of $g_{1}$: 
		
		Assume that $g_1$ is of Type $1$. In this case, the characteristic polynomial of $\bar{\rho}_{\varphi,\mfl}(\Frob_{\mfp})$ is $-\bar{\mfp}+x+x^2+x^3\in \F_{\mfl}[x]$.  Consider  three distinct ideals $\mfp_{1}=(T-c_{1})$, $\mfp_{2}=(T-c_{2})
		\in \Omega_{A}$, and $\mfp_{3} \in \Omega_A \setminus \{ (T), \mfl \}$ of degree $1$.
		Such prime ideals exist because $q \geq 7$. We can now factorize the characteristic polynomial of $\bar{\rho}_{\varphi,\mfl}(\Frob_{\mfp_{1}})$ and $\bar{\rho}_{\varphi,\mfl}(\Frob_{\mfp_{2}})$, respectively as in~\eqref{factorization_1}. By comparing their coefficients, we get
		\begin{align}\label{case_1_comparing_coeff_of_bar_rho_varphi_mfl_frob_mfl}
			\begin{cases}
				\alpha_{\mfp_1} + \zeta^{-1} \bar{\mfp}_1 = -1 = \alpha_{\mfp_2} + \zeta^{-1} \bar{\mfp}_2 \\
				\alpha_{\mfp_1}\zeta^{-1} \bar{\mfp}_1 + \zeta = 1 =  \alpha_{\mfp_2}\zeta^{-1} \bar{\mfp}_2 + \zeta
			\end{cases}
		\end{align}
		This implies that $(-1-\zeta^{-1}\bar{\mfp}_1)\bar{\mfp}_1 = (-1-\zeta^{-1}\bar{\mfp}_2)\bar{\mfp}_2\Rightarrow \zeta^{-1}(\bar{\mfp}_1^2 - \bar{\mfp}_2^2)=-(\bar{\mfp}_1-\bar{\mfp}_2)$. So $\mfp_{1}+\mfp_{2}\equiv - \zeta\pmod \mfl$. A similar argument with the pairs $(\mfp_1, \mfp_3)$, $(\mfp_2, \mfp_3)$ would imply that $\mfp_1 \equiv \mfp_2 \pmod \mfl$. This means
		$\mfl\mid(c_{2}-c_{1})$, which is a contradiction.
		
		Assume that $g_1$ is of Type $2$. In this case, the characteristic polynomial of $\bar{\rho}_{\varphi,\mfl}(\Frob_{\mfp})$ is $-\bar{\mfp}+x^2+x^3\in \F_{\mfl}[x]$ for all degree $1$ primes $\mfp = (T-c)\in \Omega_{A}\setminus\{(T),\mfl\}$. Now arguing as in the previous case, similar to \eqref{case_1_comparing_coeff_of_bar_rho_varphi_mfl_frob_mfl}, here we get
		\begin{align}\label{case_2_comparing_coeff_of_bar_rho_varphi_mfl_frob_mfl}
			\begin{cases}
				\alpha_{\mfp_1} + \zeta^{-1} \bar{\mfp}_1 = -1 = \alpha_{\mfp_2} + \zeta^{-1} \bar{\mfp}_2 \\
				\alpha_{\mfp_1}\zeta^{-1} \bar{\mfp}_1 + \zeta = 0 =  \alpha_{\mfp_2}\zeta^{-1} \bar{\mfp}_2 + \zeta
			\end{cases}
		\end{align}
		Using~\eqref{case_2_comparing_coeff_of_bar_rho_varphi_mfl_frob_mfl}, we can obtain 
		a contradiction as in the previous case.

		We now suppose that the character $\chi$ is unramified at every $\mfp\in \Omega_{A}$.  We now arguing in the above case, we get $\chi(\Frob_{\mfp})=\zeta^{\deg\mfp}$
		and hence $(\det B)(\Frob_{\mfp})=\zeta^{-\deg\mfp}\bar{\mfp}$ for every $\mfp\in \Omega_{A}\setminus\{(T),\mfl\}$. As a consequence, for any, $\mfp=(T-c)\in \Omega_{A}\setminus\{(T),\mfl\}$, we have the factorization
		\begin{align}\label{factorization_2}
			-\bar{\mfp}+g_{1}(c)^{q-1}x+x^2+x^3=(x^2-\alpha_{\mfp}x+\zeta^{-1}\bar{\mfp})(x-\zeta)\in \F_{\mfl}[x].
		\end{align}
		If $g_1$ is of Type $1$ or Type $2$, we again get a contradiction by comparing the coefficients in \eqref{factorization_2} and arguing as in the previous case. Hence, the mod-$\mfl$ representation is irreducible as required.
	\end{proof}
	
	\section{On the estimation of $| \bar{\rho}_{\varphi,\mfl}(G_{F})|$}
	Let $\mfl \in \Omega_{A}$. In this section, we shall estimate $|\bar{\rho}_{\varphi,\mfl}(I_{T})|$, 
	and hence we can get an estimate for $|\bar{\rho}_{\varphi,\mfl}(G_{F})|$.
	\begin{prop}\label{cardinality_of_A_mod_mfl_divide_im_of_mod_mfl_reduction}
		For $\mfl=(T)$, then $|A/\mfl|^{2}$ divides $|\bar{\rho}_{\varphi,\mfl}(G_{F})|$.  If $\mfl\neq (T)$, we have $|A/\mfl|^{2} = |\bar{\rho}_{\varphi,\mfl}(I_{T})|$. In particular, $|A/\mfl|^{2}$ divides $|\bar{\rho}_{\varphi,\mfl}(G_{F})|$ for all $\mfl\in\Omega_{A}$.
	\end{prop}
	
	We now use the notations and recall some facts from the proof of Lemma~\ref{Im_of_inertia_l_noteq_T}. Recall that 
	$(\psi, \Lambda)$ be the Tate datum corresponding to the Drinfeld module $\varphi$. Note that, $\{w_1,w_2\}$ is the $\F_\mfl$-basis of $\psi[\mfl]$
 and $\{e_{\Lambda}(w_1),e_{\Lambda}(w_2),e_{\Lambda}(z)\}$ is the $\F_{\mfl}$-basis of $\varphi[\mfl]$.
Since $\mfl \neq (T)$ and $I_T$ acts trivially on $w_1, w_2$,  we get $w_{1},w_2\in F_{(T)}^{\un}$, the maximal unramified extension of $F_{(T)}$, i.e., the fixed field of $I_T$. In particular, $F_{(T)}(w_1,w_2,z)\subseteq F_{(T)}^{\un}(z)$.

The valuation $\nu_{T}$ can be extended uniquely to a valuation $\nu_{T}^{\prime}$ on $F_{(T)}(w_1,w_2, z)$, a finite extension of $F_{(T)}$,
 and hence it is complete. The valuation $\nu_{T}^{\prime}$ can be extended uniquely to a valuation $\nu_{T}^{\prime\prime}$ on $F_{(T)}^{\un}(z)$. To summarise, we have
$$(F_{(T)},\nu_{T})\lra (F_{(T)}(w_1,w_2,z), \nu_{T}^{\prime})\lra (F_{(T)}^{\un}(z), \nu_{T}^{\prime\prime}).$$

To prove Proposition~\ref{cardinality_of_A_mod_mfl_divide_im_of_mod_mfl_reduction}, we need to prove two lemmas.
	
	\begin{lem}\label{four_impotant_valuations}
		Let $\lambda,z,w_{1}, w_2$ be as in the proof of  Lemma~\ref{Im_of_inertia_l_noteq_T}. Then, we have
		\begin{enumerate}
			\item $\nu_{T}^{\prime}(\lambda)<0$,
			\item $\nu_{T}^{\prime}(z)<0$,
			\item If $w\in F_{(T)}(w_1,w_2,z)$ with $\nu_{T}^{\prime}(w)\geq 0$, then $\nu_{T}^{\prime}(e_{\Lambda}(w))=\nu_{T}^{\prime}(w)$. In particular, $\nu_{T}^{\prime}(e_{\Lambda}(w_{i}))=\nu_{T}^{\prime}(w_{i})$ for $i=1, 2$,
			\item $\nu_{T}^{\prime}(e_{\Lambda}(z^{q^i}))=\nu_{T}^{\prime}(z^{q^{i}})$ for $i<2\deg_{T}(\mfl)$.
		\end{enumerate}
	\end{lem}
	\begin{proof}
		By~\cite[Example 6.2.2]{Pap23}, the first part follows. Since $\lambda=\psi_{\mfl}(z)$,
		the second part follows from the first part, as $\psi_{\mfl}(x)\in A_{(T)}[x]$.
		
		We now give a proof of the third part. Since $w\in F_{(T)}(w_1,w_2,z)$ with $\nu_{T}^{\prime}(w)\geq 0$, hence by~\eqref{e_Lambda_power_series_expression}, $e_{\Lambda}(w)=\sum_{n=0}^{\infty}u_nw^{q^n}$ converges as $\nu_{T}^{\prime}(u_nw^{q^n})\to \infty$ as $n\to \infty$. Therefore, we have 
		$\nu_{T}^{\prime}(e_{\Lambda}(w))=\lim_{k}\ \nu_{T}^{\prime}(\sum_{n=0}^{k} u_{n}w^{q^n})=\nu_{T}^{\prime}(w).$ Since $w_{1}$ and $w_{2}$ are roots of $\psi_{\mfl}(x)$, we get $\nu_{T}^{\prime}(w_{i})\geq0$. This completes the proof of the third part.
		
		We now give a proof of the fourth part. Recall that, $\Lambda$ is a $\psi$-lattice of rank $1$ with a generator $\lambda$. Now, for all $n\geq 1$, by comparing the $q^n$-th coefficient of $e_{\Lambda}(x)$, up to units of $A_{(T)}$, in \eqref{e_Lambda_power_series_expression} and \eqref{e_Lambda_for_rank_1_lattice}, we get
		$$u_{n}=\pm \displaystyle\sum_{a_{1},\ldots,a_{q^n-1}\neq 0}^{}\frac{1}{\psi_{a_{1}}(\lambda)\psi_{a_{2}}(\lambda)\cdots\psi_{a_{q^n-1}}(\lambda)},$$
		where $a_1, a_2, \ldots, a_{q^n-1} \in A\setminus\{0\}$. By taking the valuation $\nu_T^{\prime}$ on both sides, we get
		\begin{align}\label{valuation_of_u_n}
			\nu_{T}^{\prime}(u_{n})& \geq \min_{a_{1},\ldots,a_{q^n-1}\neq 0}\left\{-\sum_{j=1}^{q^n-1}\nu_{T}^{\prime}(\psi_{a_{j}}(\lambda))\right\}.
		\end{align}
		Since $\psi$ is a Drinfeld $A_{(T)}$-module of rank $2$ with good reduction at $(T)$, by~\cite[Example 6.2.2]{Pap23}, we have
		$$\nu_{T}^{\prime}(\psi_{a_{j}}(\lambda))= q^{2\deg_{T}(a_{j})}\nu_{T}^{\prime}(\lambda)=\nu_{T}^{\prime}(\lambda)+ (q^{2\dot\deg_{T}(a_{j})}-1)\nu_{T}^{\prime}(\lambda)$$
		for all $j=1,2,\ldots (q^n-1)$. Therefore, by~\eqref{valuation_of_u_n}, we get
		\begin{align*}
			\nu_{T}^{\prime}(u_{n}) \geq-(q^n-1)\nu_{T}^{\prime}(\lambda)+\min_{a_{1},\ldots,a_{q^n-1}\neq 0}\left\{- \displaystyle\sum_{j=1}^{q^n-1}(q^{2\dot\deg_{T}(a_{j})}-1)\nu_{T}^{\prime}(\lambda)\right\}.
		\end{align*}
		Since $\nu_{T}^{\prime}(\lambda)<0$, we have $\nu_{T}^{\prime}(u_{n})\geq-(q^n-1)\nu_{T}^{\prime}(\lambda)$. This implies that, for $i< 2\deg_{T}(\mfl)$, the series
		$e_{\Lambda}(z^{q^i})=\sum_{n=0}^{\infty}u_nz^{q^{i+n}}$ is convergent, by the valuation criteria, which is
		\begin{align*}
			\nu_{T}^{\prime}(u_{n}z^{q^{i+n}})&=\nu_{T}^{\prime}(u_{n})+q^{i+n}\nu_{T}^{\prime}(z)\\
			&\geq (1-q^n)\nu_{T}^{\prime}(\lambda)+q^{i+n}\nu_{T}^{\prime}(z)\\
			&=(1-q^n)q^{2\deg_{T}(\mfl)}\nu_{T}^{\prime}(z)+q^{i+n}\nu_{T}^{\prime}(z)\\
			&=q^{2\deg_{T}(\mfl)}\nu_{T}^{\prime}(z)+q^n(q^i-q^{2\deg_{T}(\mfl)})\nu_{T}^{\prime}(z)\to\infty\ \text{as}\ n\to \infty,
		\end{align*} 
		where the second equality follows from $\nu_{T}^{\prime}(\lambda)=\nu_{T}^{\prime}(\psi_{\mfl}(z))=q^{2\deg_{T}(\mfl)}\nu_{T}^{\prime}(z)$(cf. ~\cite[Example 6.2.2]{Pap23}). Hence, we have
		$$\nu_{T}^{\prime}(e_{\Lambda}(z^{q^i}))=\nu_{T}^{\prime}\left(\sum_{n=0}^{\infty}u_{n}z^{q^{i+n}}\right)=\lim_k\ \nu_{T}^{\prime}\left(\sum_{n=0}^{k} u_{n}z^{q^{i+n}}\right)
		=q^{i}\nu_{T}^{\prime}(z)<0.$$ 
		The last equality follows from, for $i< 2\deg_{T}(\mfl)$, for all $n\geq 1$, we get
		\begin{align*}
			\nu_{T}^{\prime}(u_{n}z^{q^{i+n}})=\nu_{T}^{\prime}   (u_{n})+q^{i+n}
			\nu_{T}^{\prime}(z)\geq \left(-(q^n-1)+q^{i+n-2\deg_{T}(\mfl)}\right)\nu_{T}^{\prime}(\lambda)\geq 0.
		\end{align*}
	\end{proof}
	
	\begin{lem}\label{valuation_of_linear_combination_of_basis_elements_of_psi_inverse}
		Let $\lambda,z,w_{1}, w_{2}$ be as in the proof of the Lemma~\ref{Im_of_inertia_l_noteq_T}. Let $a_{1},a_{2},b \in \F_{\mfl}$, then
		\begin{align*}
			\nu_{T}^{\prime}\left(e_{\Lambda}(a_1 w_1 + a_2 w_2 + b z)\right) =
			\begin{cases}
				q^{2i} \nu_{T}^{\prime}(z) & \text{if } b \ne 0,\ \deg_T(b) = i, \\
				\nu_{T}^{\prime}(a_1 w_1 + a_2 w_2) & \text{if } b = 0.
			\end{cases}
		\end{align*}
	\end{lem}
	\begin{proof}
		Note that if $a_1,a_2\in A$ with $a_1\equiv a_2\pmod{\mfl}$, then $a_1\cdot w=a_2\cdot w$ for any $w\in \varphi[\mfl]$. So, WLOG, we can assume $a_{1},a_{2},b\in A$ with degree less than $\deg_{T}(\mfl)$. By~\eqref{iso_drinfeld_datum}, $e_{\Lambda}$ is an $A[G_{F_{(T)}}]$-module isomorphism from $\psi_{\mfl}^{-1}(\Lambda)/\Lambda$ to $\varphi[\mfl]$, we have
		\begin{align}\label{val_e_Lambda_a1w1_a2w2_bz}
			e_{\Lambda}(a_1 w_1 + a_2 w_2 + b z)&= a_{1}\cdot e_{\Lambda}(w_1)+a_{2}\cdot e_{\Lambda}(w_2)+b\cdot e_{\Lambda}(z)\notag\\
			&=\varphi_{a_{1}}(e_{\Lambda}(w_{1}))+\varphi_{a_{2}}(e_{\Lambda}(w_{2}))+\varphi_{b}(e_{\Lambda}(z)), \ (\text{as}\ \Ima(e_{\Lambda})\subseteq \varphi[\mfl])\notag\\
			&=e_{\Lambda}(\psi_{a_{1}}(w_{1}))+e_{\Lambda}(\psi_{a_{2}}(w_{2}))+e_{\Lambda}(\psi_{b}(z)).
		\end{align}
		The last equality follows from $e_{\Lambda}\psi_{a}=\varphi_{a}e_{\Lambda}$ for all $a\in A$. Since $\nu_{T}^{\prime}(w_{j})\geq 0$ and $\psi_{T}(x)\in A_{(T)}[x]$, we have $\nu_{T}^{\prime}(\psi_{a_{j}}(w_{j}))\geq 0$. Therefore, by Lemma~\ref{four_impotant_valuations}$(3)$, we get $\nu_{T}^{\prime}(e_{\Lambda}(\psi_{a_{j}}(w_{j})))\geq 0$ for $j=1$ and $2$.
		
		Let $b\in A\setminus\{0\}$ with $\deg_{T}(b)=i<\deg_{T}(\mfl)$. Let
		$\psi_{b}(z)=d_0z+d_{1}z^q+\cdots+d_{2i}z^{q^{2i}}$ with $d_{0}=b$. Since $e_{\Lambda}(x)$ is additive, we have
		\begin{align}\label{e_Lambda_of_psi_b_of_z}
			e_{\Lambda}(\psi_{b}(z))=e_{\Lambda}(d_0z)+e_{\Lambda}(d_{1}z^q)+\cdots+e_{\Lambda}(d_{2i}z^{q^{2i}}).
		\end{align}
		By~\eqref{e_Lambda_power_series_expression}, we get
		\begin{align}\label{e_Lambda_of_d_into_z_power_q_power_j}
			e_{\Lambda}(d_jz^{q^j})=d_j(u_0z^{q^j}+u_{1}d_j^{q-1}z^{q^{j+1}}+u_{2}d_j^{q^2-1}z^{q^{j+2}}+\cdots).
		\end{align}
		Since $d_j\in A_{(T)}$ for $j=0,1,\ldots,2i$, arguing as in the proof of Lemma~\ref{four_impotant_valuations}$(4)$, we get that $e_{\Lambda}(d_jz^{q^j})$ is convergent. Taking $\nu_T^{\prime}$ on both sides of~\eqref{e_Lambda_of_d_into_z_power_q_power_j} we have
		\begin{align}\label{val_of_e_Lamb_d_j_of_z_cap_q_cap_j}
			\nu_{T}^{\prime}(e_{\Lambda}(d_jz^{q^j})) =\nu_{T}^{\prime}(d_j)+\lim_k\  \nu_{T}^{\prime}\left(\sum_{n=0}^{k}u_{n}d_j^{q^n-1}z^{q^{j+n}}\right)=\nu_{T}^{\prime}(d_j)+\nu_{T}^{\prime}(z^{q^j}).
		\end{align}
		The last equality follows from $d_j\in A_{(T)}$, and arguing as in the proof of Lemma~\ref{four_impotant_valuations}$(4)$, for $j<2\deg_{T}(\mfl)$, we get $\nu_{T}^{\prime}(u_{n}z^{q^{j+n}})\geq 0$ for all $n\geq 1$. Hence, from \eqref{e_Lambda_of_psi_b_of_z}, \eqref{val_of_e_Lamb_d_j_of_z_cap_q_cap_j}, we get
		$\nu_{T}^{\prime}(e_{\Lambda}(\psi_{b}(z))=\nu_{T}^{\prime}(e_{\Lambda}(d_{2i}z^{q^{2i}}))=\nu_{T}^{\prime}(z^{q^{2i}})=q^{2i}\nu_{T}^{\prime}(z)<0$ since $T\nmid d_{2i}$. Therefore, by taking  valuations on both sides of \eqref{val_e_Lambda_a1w1_a2w2_bz}, we get 
		$\nu_{T}^{\prime}\left(e_{\Lambda}(a_1 w_1 + a_2 w_2 + b z)\right)$ = $q^{2i}\nu_{T}^{\prime}(z)$.
		
		We now assume that $b=0$. Since $\nu_{T}^{\prime}(a_1 w_1 + a_2 w_2)\geq 0$, arguing as in the proof of Lemma~\ref{four_impotant_valuations}$(3)$, we get $e_{\Lambda}(a_1 w_1 + a_2 w_2) $ is convergent. Therefore, we have 
		\begin{align*}
		    \nu_{T}^{\prime}\left(e_{\Lambda}(a_1 w_1 + a_2 w_2 )\right)
			=\lim_k\ \nu_{T}^{\prime}\left(\sum_{n=0}^{k} u_{n}(a_1 w_1 + a_2 w_2 )^{q^n}\right) =\nu_{T}^{\prime}(a_1 w_1 + a_2 w_2).
		\end{align*}
		This completes the proof of the Lemma.    
	\end{proof}
	
    The valuation $\nu_{T}$ can also be extended uniquely to a valuation $\nu_{T}^{\un}$ on $F_{(T)}^{\un}$, which in turn can be extended uniquely to an valuation $\nu_{T}^{e_{\Lambda}}$ on $F_{(T)}^{\un}(e_{\Lambda}(z))$. Note that, we have $\nu_{T}^{e_{\Lambda}}(F_{(T)}^{\un}(e_{\Lambda}(z))^{\times})=\frac{1}{e[F_{(T)}^{\un}(e_{\Lambda}(z)):F_{(T)}^{\un}]}\Z$, where $e[F_{(T)}^{\un}(e_{\Lambda}(z)):F_{(T)}^{\un}]$ is the ramification index of $F_{(T)}^{\un}(e_{\Lambda}(z))/F_{(T)}^{\un}$. To summarise, we have
 $$(F_{(T)}, \nu_{T})\ra (F_{(T)}^{\un}, \nu_{T}^{\un}) \ra (F_{(T)}^{\un}(e_{\Lambda}(z)), \nu_{T}^{e_{\Lambda}})\ra (F_{(T)}^{\un}(z), \nu_{T}^{\prime\prime}).$$
	Now, we are in a position to prove Proposition~\ref{cardinality_of_A_mod_mfl_divide_im_of_mod_mfl_reduction}. 
    
	\begin{proof}[Proof of Proposition~\ref{cardinality_of_A_mod_mfl_divide_im_of_mod_mfl_reduction}]
		We first consider the case $\mfl=(T)$. Recall that
		\begin{align}\label{image_of_decomposition_at_T_via_mod_mfl_rep}
			\bar{\rho}_{\varphi,\mfl}(G_{F_{(T)}}) \cong G_{F_{(T)}}/\Gal\left(F^{\sep}_{(T)}/F_{(T)}(\varphi[\mfl])\right)\cong \Gal\left(F_{(T)}(\varphi[\mfl])/F_{(T)}\right),
		\end{align}
		where $G_{F_{(T)}}$ is the decomposition subgroup of $G_{F}$ at the prime $(T)$ and $F_{(T)}(\varphi[\mfl])$ is the smallest extension of $F_{(T)}$ such that $\Gal(F_{(T)}^{\sep}/F_{(T)}(\varphi[\mfl]))$ acts trivially on $\varphi[\mfl]$. For $a \in A\setminus\{0\}$, we let $\nu_{T}^{a}$ to denote the unique valuation on $F_{(T)} (\varphi[a])$ extending $\nu_{T}$. By definition, we have $\nu_{T}^{a}(F_{(T)}\left(\varphi[a])^{\times}\right)=\frac{1}{e[F_{(T)}(\varphi[a]):F_{(T)}]}\Z$, where $e[F_{(T)}(\varphi[a]): F_{(T)}]$ denotes the ramification index of $F_{(T)}(\varphi[a])/F_{(T)}$.
		
		Now by looking at the Newton's polygon of $\varphi_{T}(x)/x=T+g_{1}^{q-1}x^{q-1}+g_{2}^{q-1}x^{q^2-1}+T^{q-1}x^{q^3-1}$, we get $\varphi_{T}(x)/x$ has $q^3-q^2$ roots $\alpha \in F_{(T)}(\varphi[T])$ with valuation $\nu_{T}^{T}(\alpha)=-\frac{1}{q^2}$. Suppose $\alpha$ is one such root. Hence, $q^2=|A/\mfl|^2$ divides $e[F_{(T)}(\varphi[T]): F_{(T)}]$, in particular $q^2$ divides $|\Gal(F_{(T)}(\varphi[\mfl])/F_{(T)})|$. Therefore by~\eqref{image_of_decomposition_at_T_via_mod_mfl_rep} $q^2$ divides $|\bar{\rho}_{\varphi,\mfl}(G_F)|$.
		
		We now consider the case $\mfl\neq (T)$. Recall that
		\begin{align}\label{image_of_inertia_at_T_via_mod_mfl_rep}
			\bar{\rho}_{\varphi,\mfl}(I_{T})\cong I_{T}/\Gal(F_{(T)}^{\sep}/F_{(T)}^{\un}(\varphi[\mfl]))\cong \Gal(F_{(T)}^{\un}(\varphi[\mfl])/F_{(T)}^{\un}),
		\end{align}
		where $I_{T} :=\Gal(F^{\sep}_{(T)}/F_{(T)}^{\un})$. Since $w_{1},w_{2} \in F_{(T)}^{\un}$,  we have $F_{(T)}^{\un}(\varphi[\mfl])=F_{(T)}^{\un}(e_{\Lambda}(z))$.
		By~\eqref{expression_of_varphi_mfl_in_terms_of_e_Lambda}, we have
		$$\varphi_{\mfl}(x)=\mfl x\cdot\prod_{\substack{0\neq \gamma\prime \in \psi_{\mfl}^{-1}(\Lambda)/\Lambda}}\left(1 - \frac{x}{e_{\Lambda}(\gamma\prime)} \right).$$
		Now, by comparing the leading coefficient on both sides of the above equation, up to units of $A_{(T)}$, we  get
		$$ T^{(q - 1)\displaystyle\sum_{i = 1}^{\deg_T(\mfl)} q^{3(i - 1)}}
		= \pm  \mfl
		\prod\limits_{\substack{0 \ne \gamma' \in \psi_{\mfl}^{-1}(\Lambda)/\Lambda}}
		e_\Lambda(\gamma')^{-1}.$$
		
		By taking $\nu_{T}^{\prime}$ on both sides and using the fact that $\{w_{1}+\Lambda,w_{2}+\Lambda,z+\Lambda\}$ is a $\F_{\mfl}$-basis of $\psi^{-1}_{\mfl}(\Lambda)/\Lambda$, we get
		$$
		(q - 1)\displaystyle\sum_{i = 1}^{\deg_T(\mfl)} q^{3(i - 1)} = -\displaystyle\sum_{a_{1},a_{2},b\in \F_{\mfl}\ \text{not all zero}}^{}\nu_{T}^{\prime}(e_{\Lambda}(a_{1}w_{1}+a_{2}w_{2}+bz)).
		$$
		By Lemma~\ref{valuation_of_linear_combination_of_basis_elements_of_psi_inverse}, we have
		$$
		(q - 1)\displaystyle\sum_{i = 1}^{\deg_T(\mfl)} q^{3(i - 1)} =-(q^{2\deg_{T}(\mfl)})\left[(q-1)\displaystyle\sum_{i = 1}^{\deg_T(\mfl)}q^{3(i-1)}\right]\nu_{T}^{\prime}(z) -\nu_T^{\prime}(s),
		$$
		where $s=\prod\limits_{a_{1},a_{2}\in \F_{\mfl}\ \text{not all zero}}(a_{1}w_{1}+a_{2}w_{2})$, which is the product of all the roots of $\psi_{\mfl}(x)/x$ as well as the constant term $\mfl$ of $\psi_{\mfl}(x)/x$. Since $\nu_T^{\prime}(\mfl)=0$, we get
		$\nu_{T}^{\prime}(z)=- q^{-2\deg_{T}(\mfl)}.$ In particular, the order of $\nu_{T}^{\prime}(z)$ in $\Q/\Z$ is equal to $q^{2\deg_{T}(\mfl)}=|A/\mfl|^2$.
		Since $\nu_{T}^{e_{\Lambda}}(F_{(T)}^{\un}(e_\Lambda(z))^{\times})=\frac{1}{e[F_{(T)}^{\un}(e_\Lambda(z)):F_{(T)}^{\un}]}\Z$ and $\nu_{T}^{e_{\Lambda}}(e_{\Lambda}(z))=\nu_{T}^{\prime\prime}(e_{\Lambda}(z))=\nu_{T}^{\prime}(e_{\Lambda}(z))=\nu_{T}^{\prime}(z)$, where the last equality holds  by Lemma~\ref{four_impotant_valuations}(4)). Hence, 
		the number $q^{2\deg_{T}(\mfl)}$ divides $e[F_{(T)}^{\un}(e_\Lambda(z)):F_{(T)}^{\un}]$. By~\eqref{image_of_inertia_at_T_via_mod_mfl_rep}, we have
		$|A/\mfl|^2 \leq |\bar{\rho}_{\varphi,\mfl}(I_{T})|$. By Lemma~\ref{Im_of_inertia_l_noteq_T}, the inequality before is actually an equality. Hence, we have $|\bar{\rho}_{\varphi,\mfl}(I_{T})| = |A/\mfl|^2$.
	\end{proof}
	
	\begin{remark}\label{mfa_prime_to_T_image_of_inertia}
		The proof of Proposition~\ref{cardinality_of_A_mod_mfl_divide_im_of_mod_mfl_reduction} works for any non-zero ideal $\mfa\subseteq A$ which is prime to $(T)$. Therefore, we have $\bar{\rho}_{\varphi,\mfa}(I_{T}) = \left\{\csmat{1}{0}{b_{\sigma,1}}{0}{1}{b_{\sigma,2}}{0}{0}{1}: b_{\sigma,1}, b_{\sigma,2}\in A/\mfa\right\}$ for all non-zero ideals $\mfa \subseteq A$ which is prime to $(T)$.
	\end{remark}

	\section{On the surjectivity of mod-$\mfl$ representations $\bar{\rho}_{\varphi,\mfl}$}
	In this section, we prove the following theorem. 
	\begin{thm}\label{mod_mfl_represetation_surjective}
		Let $q\geq 7$ be an odd prime power. Let $\varphi$ be the Drinfeld  $A$-module  of rank $3$ defined by $\varphi_{T}=T+g_{1}^{q-1}\tau+g_{2}^{q-1}\tau^2+T^{q-1}\tau^3$ with $(g_{1},g_{2})\in \mathcal{G}$. The mod-$\mfl$ Galois representation
		$$\bar{\rho}_{\varphi, \mfl}:G_{F} \longrightarrow  \Aut_A(\varphi[\mfl])\cong\GL_{3}(\F_{\mfl})$$
		is surjective for all $\mfl \in \Omega_{A}$.
	\end{thm}

	Suppose $H:=\bar{\rho}_{\varphi,\mfl}(G_{F})\subseteq \GL_{3}(\F_{q^\prime})$,  where $\F_{q^\prime}  := A/\mfl$ with $q^\prime=q^{\deg_{T}{\mfl}}$. We want to show that $H= \GL_{3}(\F_{q^\prime})$.
	Suppose this is not true. Then, there exists a maximal subgroup $M$ of $\GL_{3}(\F_{q^\prime})$ containing  $H$. By Lemma~\ref{rank_r_rank_1_correspondence_with_Hayes_gives_det_of_mod_a_surejective}, we get $\det(M)=\F_{q^\prime}^{\times}$, which implies $[M:M\cap\SL_{3}(\F_{q^\prime})] = q^\prime-1$. In particular, we also get $M\cap\SL_{3}(\F_{q^\prime})$ is a proper subgroup of $\SL_{3}(\F_{q^\prime})$. If not, then $\SL_{3}(\F_{q^\prime}) \subseteq M$. Since $\det(M)=\F_{q^\prime}^{\times}$, we have $M=\GL_{3}(\F_{q^\prime})$. This is a contradiction since $M$ is proper. 
	
	We now show that $M$ contains the center $Z(\GL_{3}(\F_{q^\prime}))$. Suppose, if possible, $Z(\GL_{3}(\F_{q^\prime}))\nsubseteq M$. Now the maximality of $M$ would imply that $MZ(\GL_{3}(\F_{q^\prime}))=\GL_{3}(\F_{q^\prime})$. Since ${q^\prime}\nmid |Z(\GL_{3}(\F_{q^\prime}))|$ and $(q^\prime)^3\mid |\GL_{3}(\F_{q^\prime})|$, so we get $(q^\prime)^3\mid |M|$. Therefore, $M$ contains a Sylow $p$-subgroup of $\GL_{3}(\F_{q^\prime})$. Since the action of $H$ on $\F_{q^\prime}^3$ is irreducible, so is the action of $M$. Therefore, by \cite[Lemma 16]{Che22}, $M$ contains $\SL_{3}(\F_{q^\prime})$. By arguing as before, $M=\GL_{3}(\F_{q^\prime})$, which is a contradiction to the fact that $M$ is a proper subgroup.
	
	To summarise, we have the following information about $M$:
	\begin{enumerate}
		\item $M$ is a maximal subgroup of $\GL_{3}(\F_{q^\prime})$ that does not contains  $\SL_{3}(\F_{q^\prime})$,
		\item The action of $M$ on $\F_{q^\prime}^3$, via the mod-$\mfl$ representation $\bar{\rho}_{\varphi,\mfl}$, is irreducible,
		\item By Proposition~\ref{cardinality_of_A_mod_mfl_divide_im_of_mod_mfl_reduction}, $(q^\prime)^2\mid |M|$. Moreover, $(q^\prime)^2\mid |M\cap \SL_{3}(\F_{q^\prime})|$ (as $|M|=(q^\prime-1)|M\cap \SL_{3}(\F_{q^\prime})|$),
		\item $Z(\GL_{3}(\F_{q^\prime}))\subseteq M$.
	\end{enumerate}
	
	Recall that, we wish to show $H=\GL_{3}(\F_{q'})$. We assumed that $H$ is proper and hence
	$H$ is contained in a maximal subgroup $M$ of $\GL_{3}(\F_{q^\prime})$. By Aschbacher's Theorem~\cite[Theorem 2.2.19]{BHR13}, the maximal subgroups of $\GL_{3}(\F_{q^\prime})$ which do not contain $\SL_{3}(\F_{q'})$ are classified into $9$ classes: $8$ geometric classes $\mathcal{C}_{1},\ldots,\mathcal{C}_{8}$ and one special class $\mathcal{S}$. A brief description of these classes can also be found in~\cite[Appendix A.2]{Che22}.
	We will now show that $M$ cannot fall into these classes based on the information on $M$.
	
	$\bullet \ \mathcal{C}_{1}$: Suppose $M$ belongs to $\mathcal{C}_{1}$, then
	$M$ stabilizes a proper non-zero subspace of $\F_{q^\prime}^3$. This cannot happen since $M$ acts irreducibly on $\F_{q^\prime}^3$.
	
	$\bullet \ \mathcal{C}_{2}$: Suppose $M$ belongs to $\mathcal{C}_{2}$, then, there is a direct sum decomposition of $\F_{q^\prime}^3$ into three $1$-dimensional subspaces. Then, 
	the action of $M$ on $\F_{q^\prime}^3$ is of type $\GL_{1}(\F_{q^\prime})\wr S_{3}=\GL_{1}(\F_{q^\prime})^3\rtimes S_{3}$, the wreath product of $\GL_{1}(\F_{q^\prime})$ and the symmetric group $S_{3}$. So we have, $|M|$ divides $|\GL_{1}(\F_{q^\prime})^3\rtimes S_{3}| = (q^\prime-1)^3\cdot 3!$. This is a contradiction, since 
	$(q^\prime)^2\mid|M|$. Therefore, $M$ cannot 
	lie in $\mathcal{C}_{2}$.
	
	$\bullet \ \mathcal{C}_{3}$:  Suppose $M$ belongs to $\mathcal{C}_{3}$, then the action of $M$  on $\F_{q^\prime}^3$ is of the type $\GL_{1}(\F_{(q^\prime)^3})$. So $|M|$ divides $|\GL_{1}(\F_{(q^\prime)^3})| = \left((q^\prime)^3-1\right)$, which is a contradiction as $(q^\prime)^2\mid|M|$.
	
	$\bullet \ \mathcal{C}_{4}$:  Since there is no integer between $1$ and $\sqrt{3}$, we need not consider this case.
	
	$\bullet \ \mathcal{C}_{5}$:  Suppose $M$ belongs to $\mathcal{C}_{5}$. Then, there is a  proper  subfield $\F_{q_{0}}$ of $\F_{q^\prime}$ such that a conjugate of $M$ in $\GL_{3}(\F_{q^\prime})$ is a subgroup of $\langle Z(\GL_{3}(\F_{q^\prime})),\GL_{3}(\F_{q_{0}}) \rangle$. Note that, $q^\prime=q_{0}^d$ where $d\geq 2$.    Therefore, $|M|$ divides $|\langle Z(\GL_{3}(\F_{q^\prime})),\GL_{3}(\F_{q_{0}})\rangle|$. Now
	\begin{align*}
		|\langle Z(\GL_{3}(\F_{q^\prime})),\GL_{3}(\F_{q_{0}}) \rangle |&=\frac{|Z(\GL_{3}(\F_{q^\prime}))|\times|\GL_{3}(\F_{q_{0}})|}{|Z(\GL_{3}(\F_{q^\prime}))\cap\GL_{3}(\F_{q_{0}})|}\\
		&=\frac{(q^\prime-1)(q_{0}^3-1)(q_{0}^3-q_{0})(q_{0}^3-q_{0}^2)}{(q_{0}-1)}\\
		&=q_{0}^3(q^\prime-1)(q_{0}^3-1)(q_{0}^2-1)
	\end{align*}
	Since $(q^\prime)^{2}\mid |M|$, we get $(q^\prime)^{2}=q_{0}^{2d} \mid q_{0}^3(q^\prime-1)(q_{0}^3-1)(q_{0}^2-1)$,  which is a contradiction as $2d>3$. Therefore, $M$ cannot lie in $\mathcal{C}_{5}$.
	
	$\bullet \ \mathcal{C}_{6}$: Suppose $M$ belongs to $\mathcal{C}_{6}$. By~\cite[Page 114]{BHR13}, 
	we should have $q^{\prime}=p\equiv 1\pmod{3}$. Then there is an absolutely irreducible extraspecial $3$-group $E$ of order $3^{1+2}$ such that $E\trianglelefteq M\leq N_{\GL_{n}(q^{\prime})}(E)$, the normalizer of $E$ in $\GL_{n}(q^{\prime})$, and the action of $M$ on $\F_{q^\prime}^3$ is of type $3^{1+2}. \Sp_{2}(3)$ (cf. \cite[\S 1.2]{BHR13} for these notations). So we have $|M|$ divides $|3^{1+2}. \Sp_{2}(3)|= 2^3\cdot 3^4$, which is a contradiction, as $(q^\prime)^2\mid|M|$ and $q^\prime \geq 7$.
	
	
	$\bullet \ \mathcal{C}_{7}$:  Since there are no integers $t\geq 2$ and $m\geq 1$ such that $3=m^t$, we do not need to consider this case.
	
	$\bullet \ \mathcal{C}_{8}$: Suppose $M$ belongs to $\mathcal{C}_{8}$, then $M$ preserves a non-degenerate classical form on $\F_{q^\prime}^3$ up to scalar multiplication. By classical form, we mean symplectic form, uniform form or quadratic form:
	\begin{itemize}
		\item[(i)] Symplectic form: These forms exist only on even dimensional vector spaces. In our case, the dimension is $3$, so they do not exist.
		
		\item[(ii)]  Uniform form: For having a unitary form on a vector space over a finite field $\F_{q^\prime}$, ${q^\prime}$ needs to be a square. So assume $q^\prime=(q^{\prime\prime})^2$. Then, the action of $M$ on $\F_{q'}^3$ is of type $\GU_{3}(q^{\prime\prime})$. Therefore, $|M|$ divides $|\GU_{3}(q^{\prime\prime})|$ and by~\cite[Theorem 1.6.22]{BHR13} $|\GU_{3}(q^{\prime\prime})|=(q^{\prime\prime})^3\left(q^{\prime\prime}+1\right)\left((q^{\prime\prime})^2-1\right)\left((q^{\prime\prime})^3+1\right)$, which is a contradiction, as $(q^\prime)^2=(q^{\prime\prime})^4$ divides $|M|$.
		\item[(iii)]  Quadratic form: The action of $M$ on $\F_{q'}^3$ is of type $\GO_{3}(q^\prime)$. Therefore, $|M|$ divides $|\GO_{3}(q^\prime)|$ and by~\cite[Theorem 1.6.22]{BHR13} $|\GO_{3}(q^\prime)|=2q^\prime((q^{\prime})^2-1)$, which is a contradiction, as $(q^\prime)^2$ divides $|M|$ and $q^\prime \geq 7$.
	\end{itemize}
	Therefore, $M$ cannot lie in $\mathcal{C}_{8}$.
	
	$\bullet \ \mathcal{S}$: For this special class, we need to look at the proper subgroup $M \cap \SL_3(\F_{q^\prime})$ containing 	$Z(\SL_{3}(\F_{q^{\prime}}))$	in $\SL_3(\F_{q^\prime})$. Note that, we have the property $(q^{\prime})^2\mid |M \cap \SL_3(\F_{q^\prime})|$. Therefore, by~\cite[Theorem 4.10.2]{BHR13}, the group $M \cap \SL_3(\F_{q^\prime})$ can be any of these subgroups: $\PSL_{3}(2)\times Z(\SL_{3}(\F_{q^\prime}))$, $3^{\cdot}\textrm{A}_{6}$, $3^{\cdot}\textrm{A}_6.2_3 $, $3^{\cdot}\textrm{A}_{7}$ (cf.~\cite[\S 1.2]{BHR13} for these notations). Observe that cardinality of $\PSL_{3}(\F_2)\times Z(\SL_{3}(\F_{q^\prime}))$ is either $2^3\cdot3\cdot7$ or $2^3\cdot3^2\cdot7$ depending on $|Z(\SL_{3}(\F_{q^\prime}))|=1$ or $3$. Now the cardinality of the remaining groups is $2^3\cdot3^3\cdot5$, $2^4\cdot3^3\cdot5$, $2^3\cdot3^3\cdot5\cdot7$, respectively. Since $q^{\prime} \geq 7$ is an odd prime power and $(q^{\prime})^2\mid |M \cap \SL_3(\F_{q^\prime})|$, we get that
	$M \cap \SL_3(\F_{q^\prime})$ cannot be any of them. Hence, $M$ cannot lie in $\mathcal{S}$. 
	
	\section{On the surjectivity of $\mfl$-adic representations $\rho_{\varphi, \mfl}$}
	In this section, we prove the main theorem of this article, which is the surjectivity of $\mfl$-adic Galois representations
	attached Drinfeld $A$-modules, for $\mfl\in \Omega_{A}$.
	
	\begin{thm}[=Theorem~\ref{l_adic_surjectivity}]
		\label{Duplicate_l_adic_surjectivity}
		Let $q\geq 7$ be an odd prime power. Let $\varphi$ be a Drinfeld $A$-module  of rank $3$ defined by $\varphi_{T}=T+g_{1}^{q-1}\tau+g_{2}^{q-1}\tau^2+T^{q-1}\tau^3$ with $(g_{1},g_{2})\in \mathcal{G}$. The $\mfl$-adic Galois representation
		$$\rho_{\varphi, \mfl}:G_{F} \longrightarrow \varprojlim_{i}\ \Aut(\varphi[\mfl^i]) \cong \GL_{3}(A_{\mfl})$$
		is surjective for all $\mfl \in \Omega_{A}$.
	\end{thm}
	The proof of Theorem~\ref{l_adic_surjectivity} is based on a Proposition by Pink and R\"utsche, which we recall now. Let $\mfl$ be a finite place of $F$, i.e., $\mfl\in \Omega_A$.
	\begin{prop}[\cite{PR09}, Proposition 4.1]\label{Pink_R_cond_l_adic_sur}
		Let $H$ be a closed subgroup of $\GL_{3}(A_{\mfl})$ such that
		$\det(H)=A_{\mfl}^{\times}$. Assume that
		$|\F_{\mfl}|\geq 4$. Suppose $H\equiv\GL_{3}(\F_{l})\pmod \mfl$, and $H$ mod $\mfl^2$ contains a non-scalar element which is congruent to the identity modulo $\mfl$. Then, $H=\GL_{3}(A_{\mfl})$.
	\end{prop}
	Take $H:=\Ima(\rho_{\varphi,\mfl})$. We will show that $H$ satisfies the Proposition~\ref{Pink_R_cond_l_adic_sur}.
	Arguing as in the proof of~\cite[Proposition 4.3]{KS25}, we get
	the following proposition, which implies $\det(H)=A_{\mfl}^{\times}$.

	\begin{prop}\label{det_sur_of_l_adic_rep}
		If $\varphi$ is a Drinfeld $A$-module as in Theorem~\ref{Duplicate_l_adic_surjectivity}, then the determinant map
		$\det\rho_{\varphi,\mfl}: G_{F}\ra A_{\mfl}^{\times}$
		is surjective for all $\mfl \in \Omega_{A}$.
	\end{prop}
	
	The condition $H\equiv\GL_{3}(\F_{l})\pmod \mfl$ is equivalent to the surjectivity of mod-$\mfl$ representation $\bar{\rho}_{\varphi,\mfl}$, which is exactly the content of Theorem~\ref{mod_mfl_represetation_surjective}. We now show that $H$ mod $\mfl^2$ contains a non-scalar element congruent to the identity modulo $\mfl$. This claim will be proved in two cases.
	
	If $\mfl\neq (T)$, then by Remark~\ref{mfa_prime_to_T_image_of_inertia} $\bar{\rho}_{\varphi,\mfl^2}$ contains a non-trivial element which is congruent to the identity modulo $\mfl$, so we are done. Now, let us assume $\mfl= (T)$: Since $\varphi_{T}=T+g_{1}^{q-1}\tau+g_{2}^{q-1}\tau^2+T^{q-1}\tau^3$, we get
	\begin{align*}
		\varphi_{T^2} =& T^2 + (Tg_{1}^{q-1} + T^{q}g_{1}^{q-1})\tau + (Tg_{2}^{q-1} + g_{1}^{q^2-1}+T^{q^2}g_{2}^{q-1})\tau^2+ (T^{q}+g_{1}^{q-1}g_{2}^{q^2-q} \\
		& + g_{1}^{q^3-q^2}g_{2}^{q-1}+T^{q^3+q-1})\tau^3+(T^{q^2-q}g_{1}^{q-1}+g_{2}^{q^3-q^2+q-1}+T^{q-1}g_{1}^{q^4-q^3})\tau^4\\
		& + (T^{q^3-q^2}g_{2}^{q-1}+T^{q-1}g_{2}^{q^4-q^3})\tau^5+ T^{q^4-q^3+q-1}\tau^6
	\end{align*}
	Since $T\nmid g_{2}$, the valuation of $T^{q-1}g_{2}^{q-1}(T^{(q^2-1)(q-1)}+g_{2}^{(q^3-1)(q-1)})$, the coefficient of $\tau^5$, is $q-1$. Similarly, the valuation of the coefficient of $\tau^4$ is $0$. 
	Recall that, for any $a \in A$, $\nu_{T}^{a}$ denotes the unique valuation on $F_{(T)} (\varphi[a])$
	extending $\nu_{T}$. By the Newton's polygon of $\varphi_{T^2}(x)/x$, there are $q^5-q^4$ many roots $\alpha\in F_{(T)}(\varphi[\mfl^2])$ of $\varphi_{T^2}(x)/x$ such that $\nu_{T}^{\mfl^2}(\alpha)=-\frac{1}{q^4}$. Now, arguing as in the proof of
	Proposition~\ref{cardinality_of_A_mod_mfl_divide_im_of_mod_mfl_reduction} in the case $\mfl=(T)$, 
	we get $|\bar{\rho}_{\varphi,\mfl^2}(G_{F_{(T)}})|$ is divisible by $q^4$. The surjective group homomorphism
	$$ \bar{\rho}_{\varphi, \mfl^2}(G_F) \xrightarrow{\text{mod } \mfl} \bar{\rho}_{\varphi, \mfl}(G_F) \cong \GL_{3}(\F_{\mfl}).$$
	has non-trivial kernel $S$ in $\bar{\rho}_{\varphi, \mfl^2}(G_F)$, because $q^4$ divides $|\bar{\rho}_{\varphi,\mfl^2}(G_{F})|$, but not $|\GL_{3}(\F_{\mfl})|$. Moreover, $S$ and $\bar{\rho}_{\varphi,\mfl^2}(G_{F_{(T)}})$ have non-trivial intersection. We now show that all non-trivial elements
	in $S \cap \bar{\rho}_{\varphi,\mfl^2}(G_{F_{(T)}})\subseteq  \GL_3(A/{\mfl^2})$ cannot belong to $Z(\GL_3(A/{\mfl^2}))$. Suppose, if possible, such an element $\bar{\rho}_{\varphi,\mfl^2}(\sigma)$ is a non-trivial scalar matrix $(1+aT)I_3$
	for some $a\in \F_{q}^{\times}$, i.e., $\sigma$ acts on $\phi[\mfl^2]$ by ${1+aT}$. Let $\alpha$ be a root of $\varphi_{T^2}(x)/x$ with $\nu_{T}^{\mfl^2}(\alpha)= -\frac{1}{q^4}$. Now look at
	\begin{align*}
		\nu_{T}^{\mfl^2}(\sigma(\alpha))& =\nu_{T}^{\mfl^2}\left(\varphi_{1+aT}(\alpha)\right)\\
		& =\nu_{T}^{\mfl^2}\left([1+aT]\alpha+[ag_{1}^{q-1}]\alpha^q+[ag_{2}^{q-1}]\alpha^{q^2}+[aT^{q-1}]\alpha^{q^3}\right)\\
		& \geq \min\left\{-\frac{1}{q^4}, (q-1)\nu_{T}(g_{1})-\frac{1}{q^3},(q-1)\nu_{T}(g_{2})-\frac{1}{q^2},(q-1)-\frac{1}{q}\right\}.
	\end{align*}
	Since $(g_1,g_2) \in \mathcal{G}$, we get $\nu_{T}^{\mfl^2}(\sigma(\alpha))=-\frac{1}{q^2}$, which is a contradiction since 
	$\sigma  \in G_{F_{(T)}}$. Therefore, the matrix $\bar{\rho}_{\varphi,\mfl^2}(\sigma)$ cannot be a scalar matrix. Thus, $H$ satisfies the last condition in Proposition~\ref{Pink_R_cond_l_adic_sur}. This proves the surjectivity of the representation $\rho_{\varphi, \mfl}$.

	\section{On the surjectivity of adelic representation $\rho_{\varphi}$}
	As an application of the work done so far, in this section, we prove the surjectivity of the adelic Galois representation associated to $\varphi$. More precisely, we have:
	\begin{thm}
		\label{adelic_surjectivity}
		Let $q\geq 7$ be an odd prime power. Let $\varphi$ be a Drinfeld $A$-module of rank $3$ defined by $\varphi_{T}=T+g_{1}^{q-1}\tau+g_{2}^{q-1}\tau^2+T^{q-1}\tau^3$ with $(g_{1},g_{2})\in \mathcal{G}$. The adelic Galois representation
          $$\rho_{\varphi}:G_{F}\longrightarrow \GL_{3}(\widehat{A})$$
		is surjective.
	\end{thm}
	To prove Theorem~\ref{adelic_surjectivity}, first we need to show that $\bar{\rho}_{\varphi,\mfa}$ is surjective when $\mfa$ is the product of two distinct ideals in $\Omega_A$.
	This is achieved in a sequence of steps.
	

	\begin{lem}\label{H_GL_3_A_mod_a}
		Let $\mfl_{1}$ and $\mfl_{2}$ be two distinct elements in $\Omega_{A}$, and set $\mfa=\mfl_{1}\mfl_{2}$. Let $H=\bar{\rho}_{\varphi,\mfa}(G_{F})\subseteq \GL_{3}(A/\mfa)$. Then, the subgroup $H$ has the following properties:
		\begin{enumerate}
			\item $\det(H)=(A/\mfa)^{\times}$,
			\item For $i=1,2$, the projections $p^\prime_{i}:H^\prime\ra \SL_{3}(\F_{\mfl_{i}})$ are surjective, where $H^\prime=H\cap \SL_{3}(A/\mfa)$,
			\item the subring of $A/\mfa$ generated by the set
			$$\mathcal{S}=\left\{\tr(h)^3/\det(h)|h\in H\right\}\cup\left\{\det(h)/\tr(h)^3|h\in H\ \text{with}\ \tr(h)\in (A/\mfa)^{\times}\right\}$$
			is exactly $A/\mfa$.
		\end{enumerate}
	\end{lem}
	\begin{proof}
		By Lemma~\ref{rank_r_rank_1_correspondence_with_Hayes_gives_det_of_mod_a_surejective}, property $(1)$ follows. For $i=1,2$, the representations
		$\bar{\rho}_{\varphi,\mfl_{i}}$ are surjective, $\bar{\rho}_{\varphi,\mfl_{i}}([G_{F},G_{F}])= [\bar{\rho}_{\varphi,\mfl_{i}}(G_F), \bar{\rho}_{\varphi,\mfl_{i}}(G_F)]= [\GL_{3}(\F_{\mfl_{i}}), \GL_{3}(\F_{\mfl_{i}})]= \SL_{3}(\F_{\mfl_{i}})$. Therefore, property $(2)$ follows from $\bar{\rho}_{\varphi,\mfa} ([G_{F},G_{F}]) \subseteq H^\prime$.
		
		For property ($3$), take any $c\in \F_{q}^{\times}$ such that $\mfp=(T-c)\in \Omega_{A}\setminus\{\mfl_{1},\mfl_{2}\}$. As before, we can find the characteristic polynomial of $\bar{\rho}_{\varphi,\mfa}(\Frob_{\mfp})$ is congruent to $-\mfp+g_{1}(c)^{q-1}x+x^2+x^3$ modulo $\mfa$. Therefore
		$$\det(\bar{\rho}_{\varphi,\mfa}(\Frob_{\mfp}))/\tr(\bar{\rho}_{\varphi,\mfa}(\Frob_{\mfp}))^3\equiv -(T-c)\pmod{\mfa}.$$
		Since $q\geq 7$, there exists $c \in \F_q^{\times}$ such that $-T+c$ and $ -T+c+1$ are in $\mathcal{S}$. These elements of $\mathcal{S}$ can generate all of $A/\mfa$. Therefore, property $(3)$ follows.
	\end{proof}
	
	\begin{prop}
		\label{Normal_solvable_subgroup_of_GL_3_F_q_lies_in_its_center}
		Let $N$ be a normal solvable subgroup of $\GL_{3}(\F_{q})$, where $q$ is a prime power. Then $N\subseteq Z(\GL_{3}(\F_{q})).$
	\end{prop}
	\begin{proof}
		Let $\pi$ be the natural projection map
		$$\pi:\GL_{3}(\F_{q})\ras \PGL_{3}(\F_{q}):= \GL_{3}(\F_{q})/Z(\GL_{3}(\F_{q})).$$
		Then, the group
		$\pi(N)\cap \PSL_{3}(\F_{q})$ is either $\{1\}$ or $\PSL_{3}(\F_{q})$. In the latter case, $\PSL_{3}(\F_{q})\subseteq \pi(N)$, which is a contradiction, since $\pi(N)$ is solvable, being the image of a solvable group $N$ and therefore subgroup of it is solvable, and $\PSL_{3}(\F_{q})$ is not solvable, as it is simple and non-abelian. Therefore, $\pi(N)\cap \PSL_{3}(\F_{q})=\{1\}$. Hence, via the natural projection
		$\pi^\prime:\PGL_{3}(\F_{q})\ra \PGL_{3}(\F_{q})/\PSL_{3}(\F_{q}),$ 
		the group $\pi(N)$ injects to $\PGL_{3}(\F_{q})/\PSL_{3}(\F_{q})$.
		
	 If $|\pi(N)|>1$, then $|\pi(N)| = 3$, because $|\PGL_{3}(\F_{q})/\PSL_{3}(\F_{q})|$ $=|Z(\SL_{3}(\F_{q}))|=3$. In particular, $3|(q-1)$ and the group $\pi(N)$ is cyclic subgroup of  $\PGL_{3}(\F_{q})$ of order $3$  generated by $\bar{\delta}$, i.e., the image of $\delta$ modulo $Z( \GL_{3}(\F_{q}))$. Since $3|q-1$, $x^3-1$ is a separable polynomial that annihilates
		$\bar{\delta}$ and has roots over $\F_q$, say, $1,\omega, \omega^2$. There are $3$ choices for the minimal polynomial $\min_{\bar{\delta}}(x)$ of $\bar{\delta}$.
		\begin{itemize}
			\item Suppose $\min_{\bar{\delta}}(x)=x^3-1$. In this case, the characteristic polynomial $\ch_{\bar{\delta}}(x)$ is also $x^3-1$.
			In particular, $\bar{\delta}\in \PSL_{3}(\F_{q})$, which is a contradiction to $\pi(N)\cap \PSL_{3}(\F_{q})=\{1\}$.
			\item Suppose $\min_{\bar{\delta}}(x)=(x-\alpha)(x-\beta)$ with $\alpha, \beta\in \{1,\omega,\omega^2\}$ and $\alpha\neq \beta$. WLOG, we assume that,
			$\ch_{\bar{\delta}}(x)=(x-\alpha)(x-\beta)^2$. Then $\bar{\delta}$ similar to a diagonal matrix, say $\bar{D}$, with diagonal entries $(\alpha, \beta, \beta)$. The normalizer of $\langle\bar{D}\rangle$ is
			$$\left\{\csmat{a}{0}{0}{0}{e}{f}{0}{h}{i}: a,e,f,h,i\in \F_{q}\ \text{with} \ a(ei-fh)\in \F_{q}^{\times} \right\}\pmod {Z(\GL_{3}(\F_{q}))}.$$
			Hence, the cardinality of the normaliser of $\pi(N)$ is equal to that of the normaliser of $\langle\bar{D}\rangle$. But, the latter cardinality  $(q^2-1)(q^2-q)$ is less than $|\PGL_{3}(\F_{q})|$. This a contradiction to the normality of $\pi(N)$ in $\PGL_{3}(\F_{q})$.
			
			\item Finally, suppose $\min_{\bar{\delta}}(x)=(x-\alpha)$ for $\alpha\in \{1,\omega,\omega^2\}$. In this case, $\bar{\delta}$ is a scalar matrix,
			so $\bar{\delta} = \bar{I_3}$, which is a contradiction, as $\bar{\delta}$ is a non-trivial element of order $3$ in $\PGL_{3}(\F_{q})$.
		\end{itemize}
Hence, $|\pi(N)|>1$ cannot happen. Therefore, $|\pi(N)| = 1$, which implies $N \subseteq Z(\GL_3(\F_q))$, as required.
	\end{proof}
	Since $p \equiv 1 \pmod 3$, the above claim was easy to prove in~\cite[page 118]{Che22}. Here, we need a case-by-case analysis to prove the same for any finite field $\F_q$.

	\begin{prop}
		\label{moda-surjectivity}
		Let $H$ be a subgroup of $\GL_{3}(A/\mfa)$ satisfying the hypotheses of Lemma~\ref{H_GL_3_A_mod_a}. Then, $H=\GL_{3}(A/\mfa)$.
	\end{prop}
	\begin{proof}
		We will now show that properties $(1), (2)$ and $(3)$ in Lemma~\ref{H_GL_3_A_mod_a} imply that $H=\GL_{3}(A/\mfa)$. Let $N^\prime_{1}$ be the kernel of $p^\prime_{2}$, i.e.,
		\begin{align}
			\label{N'_1_isomorphic_to_SL_3_of_A_mod_l_i_cap_H}
			N^\prime_{1}=\left\{ h^\prime\in H^\prime| h^\prime\equiv I_{3} \pmod{\mfl_{2}}\right\}\cong H \cap (\SL_{3}(\F_{\mfl_{1}}) \times I_{3}).
	\end{align}
		Similarly, the kernel $p^\prime_{1}$ denoted by $N^\prime_{2}$ and it satisfies 
$N^\prime_{2}\cong H \cap (I_3\times \SL_{3}(\F_{\mfl_{2}}))$. By property ($2$),  $N^\prime_{i}$ is a normal subgroup of $\SL_{3}(\F_{\mfl_{i}})$, since the projections $p_i^{\prime}$ are surjective for $i=1,2$. 
By~\cite[Lemma 5.2.1]{Rib76}, the image of $H^\prime$ in $\SL_{3}(\F_{\mfl_{1}})/N^\prime_{1}\times \SL_{3}(\F_{\mfl_{2}})/N^\prime_{2}$ is the graph of the group isomorphism $\SL_{3}(\F_{\mfl_{1}})/N^\prime_{1}\xrightarrow{\sim}\SL_{3}(\F_{\mfl_{2}})/N^\prime_{2}$. In particular, this isomorphism implies that if $N^\prime_{1}=\SL_{3}(\F_{\mfl_{1}})$ then $N^\prime_{2}=\SL_{3}(\F_{\mfl_{2}})$ and vice versa. Therefore, if $N^\prime_{1}=\SL_{3}(\F_{\mfl_{1}})$, then by~\eqref{N'_1_isomorphic_to_SL_3_of_A_mod_l_i_cap_H}, we get $\SL_{3}(A/\mfa)= \SL_{3}(\F_{\mfl_{1}})\times \SL_{3}(\F_{\mfl_{2}}) \subseteq H$. Now, by property $(1)$, we have $H=\GL_{3}(A/\mfa)$.
		
		Now we show that at least one of $N^\prime_{1}=\SL_{3}(\F_{\mfl_{1}})$ or $N^\prime_{2}=\SL_{3}(\F_{\mfl_{2}})$ holds. If not, then $N^\prime_{i}$ are proper normal subgroups of $\SL_{3}(\F_{\mfl_{i}})$
		for $i=1,2$. Since $\PSL_{3}(\F_{\mfl_{i}})$ is non-abelian simple and $\SL_{3}(\F_{\mfl_{i}})$ equals to its commutator subgroup, the inclusion $N^\prime_{i}\subseteq Z(\SL_{3}(\F_{\mfl_{i}}))$ holds for $i=1,2$. Observe that $|Z(\SL_{3}(\F_{\mfl_{i}}))|$ is either $1$ or $3$, so is $|N^\prime_{i}|$. In particular, $N^\prime_{i}$ is an abelian group.

		By Theorem~\ref{mod_mfl_represetation_surjective}, the projection $p_{i}:H \ra \GL_{3}(\F_{\mfl_{i}})$ are surjective  for $i=1, 2$. Let $N_{1}$ be the kernel of $p_{2}$, i.e.,
		\begin{align}\label{N_1_isomorphic_to_GL_3_of_A_mod_l_i_cap_H}
				N_{1}=\left\{ h\in H| h\equiv I_{3} \pmod{\mfl_{2}}\right\}\cong H \cap (\GL_{3}(\F_{\mfl_{1}})\times I_{3})   
		\end{align}
		Similarly, the kernel of $p_{1}$ is denoted by $N_{2}$ and it satisfies $N_{2}\cong H \cap (I_{3}\times\GL_{3}(\F_{\mfl_{2}})$. By~\cite[Lemma 5.2.1]{Rib76}, the image of $H$ in $\GL_{3}(\F_{\mfl_{1}})/N_{1}\times \GL_{3}(\F_{\mfl_{2}})/N_{2}$  is the graph of the group isomorphism $\GL_{3}(\F_{\mfl_{1}})/N_{1}\xrightarrow{\sim}\GL_{3}(\F_{\mfl_{2}})/N_{2}$. By~\eqref{N_1_isomorphic_to_GL_3_of_A_mod_l_i_cap_H} and~\eqref{N'_1_isomorphic_to_SL_3_of_A_mod_l_i_cap_H}, we have $N_{i}\cap \SL_{3}(\F_{\mfl_{i}})=N^\prime_{i}$ and so we get the isomorphism $N_{i}/N^\prime_{i}\cong N_{i}\SL_{3}(\F_{\mfl_{i}})/\SL_{3}(\F_{\mfl_{i}})\quad \text{for}\quad i=1,2.$ Since $N_{i}\SL_{3}(\F_{\mfl_{i}})/\SL_{3}(\F_{\mfl_{i}})\subseteq \GL_{3}(\F_{\mfl_{i}})/\SL_{3}(\F_{\mfl_{i}})\cong (\F_{\mfl_{i}})^{\times}$, therefore $N_{i}/N^\prime_{i}$ is abelian. Since $N^\prime_{i}$ is abelian, $N_{i}$ is a solvable normal subgroup of $\GL_{3}(\F_{\mfl_{i}})$.
		
		By Proposition~\ref{Normal_solvable_subgroup_of_GL_3_F_q_lies_in_its_center}, $N_{i}\subseteq Z(\GL_{3}(\F_{\mfl_{i}}))$ for $i=1,2$. Now, by taking further quotient, by~\cite[Lemma 5.2.1]{Rib76}, the image of $H$ in $\PGL_{3}(\F_{\mfl_{1}})\times\PGL_{3}(\F_{\mfl_{2}})$ is the graph of a group isomorphism $\alpha:\PGL_{3}(\F_{\mfl_{1}})\xrightarrow{\sim}\PGL_{3}(\F_{\mfl_{2}})$.  Since $q$ is odd, by~\cite[Theorem 2]{Die80}, $\alpha$ can be lifted to an isomorphism	$\widetilde{\alpha}:\GL_{3}(\F_{\mfl_{1}})\xrightarrow{\sim}\GL_{3}(\F_{\mfl_{2}}).$

		Given any field isomorphism $\sigma:\F_{\mfl_{1}}\xrightarrow{\sim}\F_{\mfl_{2}}$ and a character $\chi:\GL_{3}(\F_{\mfl_{1}})\ra \F_{\mfl_{2}}^{\times}$, we can create two isomorphisms $\GL_{3}(\F_{\mfl_{1}})\xrightarrow{\sim} \GL_{3}(\F_{\mfl_{2}})$ by 
		\begin{itemize}
			\item[(first type)] $B\mapsto \chi(B)gB^{\sigma}g^{-1}$
			\item[(second type)] $B\mapsto \chi(B)g((B^{T})^{-1})^{\sigma}g^{-1}$
		\end{itemize}
		where $B\in\GL_{3}(\F_{\mfl_{1}})$, $B^{\sigma}$ is the matrix that applies $\sigma$ to each entry of $B$, and $g\in\GL_{3}(\F_{\mfl_{2}})$. Such a $\sigma$ can exist because $|\F_{\mfl_{1}}|=|\F_{\mfl_{2}}|$. This follows from
$\SL_{3}(\F_{\mfl_{1}})/N^\prime_{1}\xrightarrow{\sim}\SL_{3}(\F_{\mfl_{2}})/N^\prime_{2}$. By~\cite[Theorem 1]{Die80},  $\widetilde{\alpha}$ must be one of the above two types for some $\sigma,\chi$ and $g$.
		
		\begin{lem}\label{tilde_alpha_must_be_of_first_type}
			The map $\widetilde{\alpha}$ must be of the first type.
		\end{lem}
		\begin{proof}
			Suppose, if possible, the map $\widetilde{\alpha}$ is of the second type. Let $\mfp=(T-c)\in \Omega_{A}\setminus\{\mfl_{1},\mfl_{2},(T)\}$. The projection of the image of $\bar{\rho}_{\varphi,\mfa}(\Frob_{\mfp})\in H$ in $\PGL_{3}(\F_{\mfl_{1}})\times\PGL_{3}(\F_{\mfl_{2}})$ onto the second factor is given by $\bar{\rho}_{\varphi,\mfl_{2}}(\Frob_{\mfp})\cdot Z(\GL_{3}(\F_{\mfl_{2}}))$, which is $\widetilde{\alpha}(\bar{\rho}_{\varphi,\mfl_{1}}(\Frob_{\mfp}))\cdot Z(\GL_{3}(\F_{\mfl_{2}}))= g(((\bar{\rho}_{\varphi,\mfl_{1}}(\Frob_{\mfp}))^{T})^{-1})^{\sigma}g^{-1}\cdot Z(\GL_{3}(\F_{\mfl_{2}}))$.
			Hence, the elements  $\bar{\rho}_{\varphi,\mfl_{2}}(\Frob_{\mfp})$ and $g(((\bar{\rho}_{\varphi,\mfl_{1}}(\Frob_{\mfp}))^{T})^{-1})^{\sigma}g^{-1}$ are in the same coset of $\PGL_{3}(\F_{\mfl_{2}})$, i.e., there exist a $\lambda\in \F_{\mfl_{2}}^{\times}$ such that
			\begin{align}\label{same_coset_of_PGL_3_A_mod_mfl_2}
				g(((\bar{\rho}_{\varphi,\mfl_{1}}(\Frob_{\mfp}))^{T})^{-1})^{\sigma}g^{-1}=(\lambda\cdot I_{3})\cdot\bar{\rho}_{\varphi,\mfl_{2}}(\Frob_{\mfp}).
			\end{align}
			Now computing as in Theorem~\ref{varphi_mfl_irreducible_F_mfl_G_F_module}, the characteristic polynomials of $(\bar{\rho}_{\varphi,\mfl_{1}}(\Frob_{\mfp}))^{-1}$ and $\bar{\rho}_{\varphi,\mfl_{2}}(\Frob_{\mfp})$  are $- {\mfp^{-1}}-{\mfp^{-1}}x-{g_{1}(c)^{q-1}}{\mfp^{-1}} x^2+x^3\pmod{\mfl_{1}}$ and $-\mfp+g_{1}(c)^{q-1}x+x^2+x^3\pmod{\mfl_{2}}$, resp.
			
			If $g_1$ is of Type 1, then we compare the trace and determinant on both sides of~\eqref{same_coset_of_PGL_3_A_mod_mfl_2} to get
			$$ -\lambda \equiv \sigma({\mfp}^{-1}\pmod{\mfl_{1}})\equiv\lambda^{3}\mfp\pmod{\mfl_{2}}.$$
			Therefore, $\mfp\equiv-(\frac{1}{\lambda})^2\pmod{\mfl_{2}}$. This congruence cannot hold if $\deg_{T}(\mfl_{2})\geq 2$. If $\mfl_{2}=(T-c_{2})$ for some $c_{2}\in \F_{q}$, then $c-c_{2}=(\frac{1}{\lambda})^2$ in $\F_{q}^{\times}$. Since $q \geq 7$, we can choose $c$ so that $c-c_2 \not \in (\F_q^{\times})^2$. This gives us a contradiction. 
			
			If $g_1$ is of Type 2, then the characteristic polynomials of $\bar{\rho}_{\varphi,\mfl_{2}}(\Frob_{\mfp})$, $(\bar{\rho}_{\varphi,\mfl_{1}}(\Frob_{\mfp}))^{-1}$ are $x^3+x^2-\mfp\pmod{\mfl_{2}}$, $x^3-{\mfp^{-1}}x- {\mfp^{-1}}\pmod{\mfl_{1}}$, resp., By~\eqref{same_coset_of_PGL_3_A_mod_mfl_2}, the trace of            $g(((\bar{\rho}_{\varphi,\mfl_{1}}(\Frob_{\mfp}))^{T})^{-1})^{\sigma}g^{-1}$ is $0$ and the trace of $(\lambda\cdot I_{3})\cdot\bar{\rho}_{\varphi,\mfl_{2}}(\Frob_{\mfp})$ is non-zero. This gives us the desired contradiction.
			Therefore, $\widetilde{\alpha}$ must be of the first type.
		\end{proof}
		By Lemma~\ref{tilde_alpha_must_be_of_first_type}, $\widetilde{\alpha}(B)=\chi(B)gB^{\sigma}g^{-1}$ for all $B\in\GL_{3}(\F_{\mfl_{1}})$. Hence, we have $\frac{\tr(\widetilde{\alpha}(B)^3)}{\det(\widetilde{\alpha}(B))}=\sigma\left(\frac{\tr(B)^3}{\det(B)}\right).$ Define $W:=\{(x_{1},x_{2})|\sigma(x_{1})=x_{2}\}$ be a subring of $A/\mfa\cong \F_{\mfl_{1}}\times \F_{\mfl_{1}}$. For each element $(h_{1},h_{2})\in H$, we have $\widetilde{\alpha}(h_1)$ and $h_2$ in the same quotient of $\PGL_{3}(\F_{\mfl_2})$, i.e., $\widetilde{\alpha}(h_1)=\lambda h_2$ for some $\lambda\in \F_{\mfl_2}^{\times}$. Then, we have $$\sigma\left(\frac{\tr(h_1)^3}{\det(h_1)}\right)=\frac{\tr(\widetilde{\alpha}(h_1)^3)}{\det(\widetilde{\alpha}(h_1))}=\frac{\tr(h_2)^3}{\det(h_2)}.$$
		In particular, we have $\left(\frac{\tr(h_{1})^3}{\det(h_{1})},\frac{\tr(h_{2})^3}{\det(h_{2})}\right) \in W$. Also, the element $\left(\frac{\det(h_{1})}{\tr(h_{1})^3}, \frac{\det(h_{2})}{\tr(h_{2})^3}\right)\in W$ if $\tr(h_{1})\neq 0$ and $\tr(h_{2})\neq 0$. Therefore, we have $\mathcal{S}\subseteq W$, and hence $A/\mfa \subseteq W$ by Property ($3$). This is a contradiction since $|W| < |A/\mfa| $.  Hence, at least one of the $N^\prime_{1}, N^\prime_{2} $ is not proper. Therefore, $H=\GL_{3}(A/\mfa)$, as required.
	\end{proof}
	
	Finally, we are in a position to give a proof of Theorem~\ref{adelic_surjectivity}.
	
	\begin{proof}[Proof of Theorem~\ref{adelic_surjectivity}]
		By Proposition~\ref{Hayes}, the adelic representation of the Carlitz module is surjective. Arguing as in Lemma~\ref{rank_r_rank_1_correspondence_with_Hayes_gives_det_of_mod_a_surejective}, we get 
		$\det\rho_{\varphi}(G_{F})=\rho_{C}(G_{F})={\widehat{A}}^{\times}$. Now, it is enough to show $\rho_{\varphi}([G_{F}, G_{F}])=\SL_{3}(\widehat{A})$ and this is equivalent to  $\bar{\rho}_{\varphi,\mfa}([G_{F},G_{F}])=\SL_{3}(A/\mfa)\cong \prod_{i}^{}\SL_{3}(A/{\mfl_{i}^{n_{i}}}), $ for every ideal $0 \neq \mfa=\mfl_{1}^{n_{1}}\mfl_{2}^{n_{2}}\cdots \mfl_{k}^{n_{k}}$. By~\cite[Lemma 20]{Che22}, each $\SL_{3}(A/{\mfl_{i}^{n_{i}}})$ has no non-trivial abelian quotient. Arguing as in the proof of~\cite[Lemma 24]{Che22}, each projection
		$$\bar{\rho}_{\varphi,\mfa}:[G_{F},G_{F}]\longrightarrow \SL_{3}(A/{\mfl_{i}^{n_{i}}})\times \SL_{3}(A/{\mfl_{j}^{n_{j}}})$$ is surjective for $1\leq i\leq j\leq k$.  Now, by~\cite[Lemma 5.2.2]{Rib76}, we have 
		$\bar{\rho}_{\varphi,\mfa}([G_{F},G_{F}])=\SL_{3}(A/\mfa)$ for all non-zero ideal $\mfa$ of $A$. This completes the proof of the Theorem.
	\end{proof}

	\section{Comparisons with~\cite{Che22}}
	We conclude the article with remarks highlighting comparisons with the results presented in~\cite{Che22}.
	\begin{itemize}
		\item  In this article, our results hold over any $\F_q$ with $q=p^e \geq 7$ be an odd prime power, whereas~\cite{Che22} restricts to primes $p \in \mathbb{P}$ satisfying $p \equiv 1\pmod{3}$.

    	\item In Theorem~\ref{adelic_surjectivity}, we construct a two-parameter family of Drinfeld modules with surjective adelic Galois representations. This family includes the specific Drinfeld module studied in~\cite{Che22}.
		
		\item In $\S3$, to prove the irreducibility of $\varphi[\mfl]$ as a $\F_{\mfl}[G_{F}]$-module for $\mfl=(T)$ case by contradiction, we work with the characteristic polynomials of $\bar{\rho}_{\varphi, T}(\Frob_{(T-c)})$ for $c\in \F_{q}^{\times}$. We use the theory of permutation polynomials to get a contradiction. In contrast, the corresponding step in~\cite{Che22} is more straightforward, as the characteristic polynomials involved are simpler.

		\item We establish that Proposition~\ref{Normal_solvable_subgroup_of_GL_3_F_q_lies_in_its_center} holds
		over an arbitrary finite field. Previously, in~\cite{Che22}, the same conclusion was obtained under the assumption that the prime $p$ with $p \equiv 1\pmod{3}$.

		\item In the proof of Lemma~\ref{tilde_alpha_must_be_of_first_type}, our approach required an additional argument to demonstrate that two matrices do not lie in the same coset of $\PGL_{3}(\F_{\mfl_{2}})$, as both matrices possess non-zero traces. In contrast, the argument in~\cite{Che22} is more straightforward, relying on the fact that one matrix has trace zero while the other has non-zero trace.
   \end{itemize}


  \bibliographystyle{plain, abbrv}
	
\end{document}